\pdfoutput=1
\documentclass{amsart}

\usepackage{
 amsmath, 
 amsxtra, 
 amsthm, 
 amssymb, 
 etex, 
 mathrsfs, 
 mathtools, 
 tikz-cd, 
 xr}
\usepackage[all]{xy}
\usepackage{hyperref}

\newtheorem{theorem}{Theorem}[subsection]
\newtheorem{lemma}[theorem]{Lemma}
\newtheorem{conjecture}[theorem]{Conjecture}
\newtheorem{proposition}[theorem]{Proposition}
\newtheorem{corollary}[theorem]{Corollary}
\newtheorem{defn}[theorem]{Definition}

\newtheorem{lthm}{Theorem} % theorems with letters (for intro)

\theoremstyle{remark}
\newtheorem{remark}[theorem]{Remark}
\newtheorem{conv}[theorem]{Convention}
\newcommand{\QQ}{\mathbb{Q}}
\newcommand{\Qp}{\mathbb{Q}_p}
\newcommand{\Qpn}{\QQ_{p,n}}
\newcommand{\ZZ}{\mathbb{Z}}
\newcommand{\Zp}{\mathbb{Z}_p}

\newcommand{\NN}{\mathbb{N}}
\newcommand{\DD}{\mathbb{D}}
\renewcommand{\AA}{\mathbb{A}}
\newcommand{\Brig}{\mathbb{B}_{\rig,\Qp}^+}
\newcommand{\AQp}{\AA_{\Qp}^+}
\newcommand{\Dcris}{\DD_{\mathrm{cris}}}

\newcommand{\TT}{\mathbb{T}}
\newcommand{\VV}{\mathbb{V}}

\newcommand{\cBF}{\mathcal{BF}}
\newcommand{\BF}{\textup{BF}}
\newcommand{\wBF}{\widetilde{\BF}}

\DeclareMathOperator{\Gal}{Gal}
\DeclareMathOperator{\Fil}{Fil}
\DeclareMathOperator{\Hom}{Hom}
\DeclareMathOperator{\Sel}{Sel}
\DeclareMathOperator{\coker}{coker}

\newcommand{\ord}{\mathrm{ord}}
\newcommand{\vp}{\varphi}
\newcommand{\Iw}{\mathrm{Iw}}
\newcommand{\HIw}{H^1_{\Iw}}
\newcommand{\col}{\mathrm{Col}}
\newcommand{\image}{\mathrm{Im}}
\newcommand{\loc}{\mathrm{loc}}
\newcommand{\Char}{\mathrm{char}}
\newcommand{\LL}{\Lambda}
\newcommand{\lra}{\longrightarrow}
\newcommand{\res}{\textup{res}}
\newcommand{\adj}{\mathrm{adj}}
\newcommand{\pr}{\textup{pr}}
\newcommand{\cor}{\mathrm{cor}}
\newcommand{\Mlog}{M_{\log}}
\newcommand{\Tw}{\mathrm{Tw}}
\newcommand{\lb}{[[}
\newcommand{\rb}{]]}
\newcommand{\uCol}{\underline{\col}}
\newcommand{\rig}{\mathrm{rig}}
\newcommand{\an}{\mathrm{an}}

\newcommand{\fM}{\mathfrak{M}}
\newcommand{\fS}{\mathfrak{S}}
\newcommand{\fL}{\mathfrak{L}}
\newcommand{\fn}{\mathfrak{n}}

\newcommand{\cF}{\mathcal{F}}
\newcommand{\cG}{\mathcal{G}}
\newcommand{\cH}{\mathcal{H}}
\newcommand{\cL}{\mathcal{L}}
\newcommand{\cO}{\mathcal{O}}
\newcommand{\cS}{\mathcal{S}}
\newcommand{\cT}{\mathcal{T}}
\newcommand{\cV}{\mathcal{V}}
\newcommand{\NP}{\mathcal{N}(\mathcal{P})}

\newcommand{\bc}{\bullet,\circ}
\newcommand{\ts}{\triangle,\square}
\newcommand{\cFbc}{\cF_{\bc}}

\definecolor{Green}{rgb}{0.0, 0.5, 0.0}

\newcommand{\draftcolor}{Green}

\newcommand{\bgreen}{\begin{color}{\draftcolor}}
\newcommand{\egreen}{\end{color}}

\begin{document}

\title[Non-Ordinary Rankin--Selberg Products]{Iwasawa theory for Rankin--Selberg products of $p$-non-ordinary eigenforms}

\begin{abstract}
 Let $f$ and $g$ be two modular forms which are non-ordinary at $p$. The theory of Beilinson--Flach elements gives rise to four rank-one non-integral Euler systems for the Rankin--Selberg convolution $f \otimes g$, one for each choice of $p$-stabilisations of $f$ and $g$. We prove (modulo a hypothesis on non-vanishing of $p$-adic $L$-functions) that the $p$-parts of these four objects arise as the images under appropriate projection maps of a single class in the wedge square of Iwasawa cohomology, confirming a conjecture of Lei--Loeffler--Zerbes.

 Furthermore, we define an explicit logarithmic matrix using the theory of Wach modules, and show that this describes the growth of the Euler systems and $p$-adic $L$-functions associated to $f \otimes g$ in the cyclotomic tower. This allows us to formulate ``signed'' Iwasawa main conjectures for $f\otimes g$ in the spirit of Kobayashi's $\pm$-Iwasawa theory for supersingular elliptic curves; and we prove one inclusion in these conjectures under our running hypotheses.
\end{abstract}

\author[K. B\"uy\"ukboduk]{K\^az\i m B\"uy\"ukboduk}
\address[B\"uy\"ukboduk]{UCD School of Mathematics and Statistics\\ University College Dublin\\ Ireland}
\email{kazim.buyukboduk@ucd.ie}

\author[A. Lei]{Antonio Lei}
\address[Lei]{D\'epartement de Math\'ematiques et de Statistique\\
Universit\'e Laval, Pavillion Alexandre-Vachon\\
1045 Avenue de la M\'edecine\\
Qu\'ebec, QC\\
Canada G1V 0A6}
\email{antonio.lei@mat.ulaval.ca}

\author[D.~Loeffler]{David Loeffler}
\address[Loeffler]{Mathematics Institute\\
Zeeman Building, University of Warwick\\
Coventry CV4 7AL, UK}
\email{d.a.loeffler@warwick.ac.uk}

\author[G.~Venkat]{Guhan Venkat}
\address[Venkat]{D\'epartement de Math\'ematiques et de Statistique\\
Universit\'e Laval, Pavillion Alexandre-Vachon\\
1045 Avenue de la M\'edecine\\
Qu\'ebec, QC\\
Canada G1V 0A6}
\email{guhanvenkat.harikumar.1@ulaval.ca}

\numberwithin{equation}{subsection}

\thanks{The authors' research is partially supported by: the Turkish Academy of Sciences, T\"UB\.ITAK grant 113F059 and a EC Global Fellowship (B\"uy\"ukboduk); the NSERC Discovery Grants Program 05710 (Lei, Venkat); Royal Society University Research Fellowship ``$L$-functions and Iwasawa theory'' (Loeffler); a CRM-Laval postdoctoral fellowship and a Philip Leverhulme Prize grant PLP-2014-354 funded by Leverhulme Trust (Venkat).}

\maketitle

\setcounter{tocdepth}{1}
\tableofcontents

\section{Introduction}

\subsection{The setting} Throughout this article, we fix an odd prime $p$ and embeddings $\iota_\infty:\overline{\QQ}\hookrightarrow \mathbb{C}$ and $\iota_p:\overline{\QQ}\hookrightarrow \mathbb{C}_p$. Let  $f$ and $g$ be two normalised, new cuspidal modular eigenforms, of weights $k_f + 2, k_g + 2$, levels $N_f, N_g$, and characters $\epsilon_f$, $\epsilon_g$ respectively. We assume that $k_f, k_g \ge 0$, that $p \nmid N_f N_g$, and that $f$ and $g$ are both non-ordinary at $p$ (with respect to the embeddings we fixed).

Let $E/\Qp$ be a finite extension containing the coefficients of $f$ and $g$, as well as the roots of the Hecke polynomials of $f$ and $g$ at $p$. We shall write $\alpha_f$, $\beta_f$, $\alpha_g$ and $\beta_g$ for these roots; we assume throughout that $\alpha_f \ne \beta_f$ and $\alpha_g \ne \beta_g$.

Let $\cO$ denote the ring of integers of $E$. For each $h\in\{f,g\}$, we fix a Galois-stable $\cO$-lattice $R_h$ inside Deligne's $E$-linear representation of $G_\QQ$. The goal of this article is to study the Iwasawa theory of $T \coloneqq R_f^*\otimes R_g^*=\Hom(R_f\otimes R_g,\cO)$ over $\QQ(\mu_{p^\infty})$.

\subsection{Main results}

 \subsubsection*{Beilinson--Flach elements} For each of the four choices of pairs $(\lambda, \mu)$, where $\lambda \in \{ \alpha_f, \beta_f\}$ and $\mu\in\{\alpha_g,\beta_g\}$, and each integer $m\ge1$ coprime to $p$, there exists a  Beilinson--Flach class
\[
\BF_{\lambda,\mu,m}\in \cH\otimes\HIw(\QQ(\mu_{mp^\infty}),T),
\]
 as constructed in \cite{LZ1}. For the definition of $\cH$ and $\HIw$, see \S\ref{S:review} below. The classes $\BF_{\lambda,\mu,m}$ satisfy Euler-system norm relations as $m$ {varies}. However, they are not integral; that is, they do not lie $\HIw(\QQ(\mu_{mp^\infty}),T)$. This is the chief difficulty in using these elements to study the Iwasawa theory of $T$. 
 
 In \S \ref{sect:BFelt} below, we recall the relevant properties of these classes, focussing on their dependence on the choice of the $p$-stabilisation data $(\lambda, \mu)$. We also recall the explicit reciprocity law relating the Beilinson--Flach classes for $m = 1$ to $p$-adic Rankin--Selberg $L$-functions: applying the Perrin-Riou regulator map to the four classes $\BF_{\lambda,\mu,1}$ and projecting to suitable eigenspaces, we obtain the four unbounded $p$-adic $L$-functions associated to $f$ and $g$ studied in \cite{LZ1}. 
 
 As a by-product of this analysis we also obtain four new $p$-adic $L$-functions, which are defined and studied in \S\ref{sect:extraL}. We conjecture that these fall into two pairs, with each pair differing only by a sign. This gives a total of 6 $p$-adic $L$-functions for $f$ and $g$, which is consistent with a conjecture of Perrin-Riou, predicting one $p$-adic $L$-function for every $\varphi$-eigenspace in the 6-dimensional space $\bigwedge^2_E \Dcris(V)$.
 
 \subsubsection*{A ``rank 2'' Beilinson--Flach element} In \cite{LLZ1}, it was conjectured that the four Beilinson--Flach elements associated to $f$ and $g$ can be seen as the images, under suitable linear functionals, of a single element in the wedge square of Iwasawa cohomology. We recall this conjecture (in a slightly strengthened form) as Conjecture \ref{conj:rank2} below. Our first main result gives a partial confirmation of this conjecture, assuming $m = 1$ and some technical {hypotheses}:
 
 \begin{lthm}
 \label{thm:liftingBFattrivialtamelevelinFracHINTRO}
 Suppose that all four $p$-adic Rankin--Selberg $L$-functions given as in \eqref{eq:fourLp} are non-zero-divisors in $\cH$, and that the following ``big image'' hypotheses hold:
 \begin{itemize}
  \item The representation $V$ is absolutely irreducible.
  \item There exists an element $\tau\in \Gal(\overline{\QQ}/\QQ(\mu_{p^\infty}))$ such that the $E$-vector space $V/(\tau-1)V$ is 1-dimensional.
 \end{itemize}
 
 Then there exists a class $\BF_1 \in \operatorname{Frac} \cH \otimes_{\Lambda} \bigwedge^2 H^1_{\Iw}(\QQ(\mu_{p^\infty}), T)$ whose image under the appropriate choice of Perrin-Riou functional $($corresponding to the choice $
 \lambda \in \{\alpha_f,\beta_f\}$ and $\mu \in \{\alpha_g,\beta_g\}$$)$ equals $ \BF_{\lambda, \mu, 1}$.  % In particular the ``extra'' anti-symmetry property \eqref{eq:extraLsym} holds.
\end{lthm}

 See Theorem~\ref{thm:liftingBFattrivialtamelevelinFracH} below for a more precise formulation of Theorem~\ref{thm:liftingBFattrivialtamelevelinFracHINTRO}. 

\subsubsection*{Decompositions using matrices of logarithms} By analogy with earlier work on Iwasawa theory of $p$-supersingular motives (such as \cite{kobayashi03, lei09,leiloefflerzerbes10,leiloefflerzerbes11, BL16b}), one naturally expects the growth of the denominators of the Beilinson--Flach elements and $p$-adic $L$-functions for $f \otimes g$ to be governed by a suitable ``matrix of logarithms'', depending only on the restriction of $T$ to the decomposition group at $p$.

In this paper, we construct such a logarithmic matrix for the representation $T$. More precisely, we show that there exists a $4\times 4$ matrix $M$ defined over $\cH$ allowing us to decompose Perrin-Riou's regulator map $\cL:\HIw(\Qp,T)\rightarrow \cH\otimes \Dcris(T)$. That is, there exist four bounded Coleman maps 
$$\col_{\bc}:\HIw(\Qp,T)\lra \cO[[\Gal(\Qp(\mu_{p^\infty})/\Qp)]]\,\,,\,\,\bc\in\{\#,\flat\}$$ 
such that
\[
\cL=\begin{pmatrix}
v_1&v_2&v_3&v_4
\end{pmatrix}\cdot M\cdot \begin{pmatrix}
\col_{\#,\#}\\
\col_{\#,\flat}\\
\col_{\flat,\#}\\
\col_{\flat,\flat}
\end{pmatrix}
\]
for some basis $\{v_1,v_2,v_3,v_4\}$ of $\Dcris(T)$. %Similar to the setting of a single modular form, the kernels of these Coleman maps allow us to define local Selmer conditions at $p$.

We conjecture that this logarithmic matrix can be used to decompose the unbounded Beilinson--Flach elements $\BF_{\lambda, \mu, m}$ into bounded classes. The precise formulation (Conjecture \ref{conj:signedfactBF} below) is that there should exist elements $\BF_{\bc,m}\in\HIw(\QQ(\mu_{mp^\infty}),T)$ for $\bc\in\{\#,\flat\}$ such that
\[
\begin{pmatrix}
\BF_{\alpha,\alpha,m}\\
\BF_{\alpha,\beta,m}\\
\BF_{\beta,\alpha,m}\\
\BF_{\beta,\beta,m}
\end{pmatrix}
=M \cdot 
\begin{pmatrix}
\BF_{\#,\#,m}\\
\BF_{\#,\flat,m}\\
\BF_{\flat,\flat,m}\\
\BF_{\flat,\flat,m}
\end{pmatrix}.
\tag{\dag}
\]
We refer to these conjectural $\BF_{\bc, m}$ as \emph{doubly-signed Beilinson--Flach elements}, since they depend on the two choices of symbols $(\bc)$. 

While we are currently unable to prove such a decomposition, we show a partial result in this direction in \S\ref{subsec:partialsignedBFsplitting}, where we consider the images of the unbounded Beilinson--Flach elements at locally-algebraic characters in a certain range. We also show that if our Conjecture~\ref{conj:rank2} on the existence of integral rank-two classes $\{\BF_m\}$ holds for some $m$, then Conjecture \ref{conj:signedfactBF} follows as a consequence.

\subsubsection*{Bounded $p$-adic $L$-functions and Iwasawa main conjectures} On assuming that integral classes $\BF_{\bc,1}$ exist satisfying ($\dag$) for $m = 1$, we may define bounded $p$-adic $L$-functions by applying the integral Coleman maps $\col_{\ts}$ to these elements. Here $(\ts) \in \{\#, \flat\}^2$ is a second pair of symbols; we refer to these $p$-adic $L$-functions as \emph{quadruply-signed}. This gives a factorization of the unbounded $p$-adic $L$-functions into bounded ones. We formulate an Iwasawa main conjecture (Conjecture~\ref{conj:quadruplysignedmainconj} below) relating these $p$-adic $L$-functions to the characteristic ideals of the Selmer groups defined using the the intersection of the kernels of $\col_{\bc}$ and $\col_{\ts}$. Using  the classical Euler system machine,  we are able to show that one inclusion of this main conjecture holds under various technical hypotheses:

\begin{lthm}
 \label{thm:qsignedmainconjINTRO}
 Suppose that $|k_f-k_g| \geq 3$ and  $p>k_f + k_g + 2$. Assume the validity of Conjecture~\ref{conj:signedfactBF}, and of the hypotheses  {\bf{(A--Sym)}}, $\mathbf{(H.nA)}, \mathbf{(BI0)}$ and at least one of $\mathbf{(BI1)}-\mathbf{(BI2)}$ stated in \S\ref{sect:quadsignedMC} below. For any integer $j$ with $1+\frac{k_f+k_g}{2}<j\leq \max(k_f, k_g)$, there exists at least one choice of symbols $\fS=\{(\ts),(\bc)\}$ with $(\ts),(\bc)\in \{\#,\flat\}^2$ such that the $\omega^j$-isotypic component of the quadruply-signed Selmer group $\Sel_{\fS}(T^\vee(1)/\QQ(\mu_{p^\infty}))$ is $\cO[[\Gamma_1]]$-cotorsion and 
 $$e_{\omega^j}\fL_{\fS}\in \Char_{\cO[[\Gamma_1]]}\left(e_{\omega^j} \Sel_{\fS}(T^\vee(1)/\QQ(\mu_{p^\infty}))^\vee\right)$$
 as ideals of $\cO[[\Gamma_1]]\otimes \Qp$. 
\end{lthm}

See Definition~\ref{defn:quad} below where we define the quadruply-signed Selmer group $\Sel_{\fS}(T^\vee(1)/\QQ(\mu_{p^\infty}))$ and the quadruply-signed $p$-adic $L$-function $\fL_{\fS}$. See also Theorem~\ref{thm:qsignedmainconj}  for a more precise formulation of Theorem~\ref{thm:qsignedmainconjINTRO} as well as Proposition~\ref{prop:antisymmetryofmatrixofsignedpadicLfunctions} where we give a sufficient condition for the validity of the condition {\bf{(A--Sym)}}.

\subsubsection*{Triangulordinary Selmer groups} An alternative approach to Iwasawa theory for supersingular motives is given by Pottharst's theory of triangulordinary Selmer groups. This allows us to associate a Selmer group to each of the six $\varphi$-eigenspaces in $\Dcris(V)$; these Selmer groups are finitely-generated over the distribution algebra $\cH$, rather than the Iwasawa algebra, and one expects their characteristic ideals to be generated by the associated unbounded $p$-adic $L$-functions. As a consequence of our results on the Iwasawa main conjectures for the bounded, quadruply-signed $p$-adic $L$-functions, we obtain one inclusion in the Pottharst-style main conjectures, under our running hypotheses:

\begin{lthm}[Corollary~\ref{cor:analyticmainconjectureimproved} below]
 \label{thm:analyticmainconjectureimprovedINTRO}
 Suppose that all hypotheses in Theorem~\ref{thm:qsignedmainconjINTRO} hold true and choose $\fS=\{(\ts),(\bc)\}$ that ensures the validity of the conclusions of Theorem~\ref{thm:qsignedmainconjINTRO}. Then for each $\lambda\in \{\alpha_f,\beta_f\}$ we have the divisibility
 $$\Char\left(H^2(\QQ, \VV^\dagger; \DD_{\lambda})\right)\, \big{|}\, \Char\left(\coker \col_{\bc}\right) L_p(f_\lambda, g)$$
 taking place in the ring $\cH$. Here, $H^2(\QQ, \VV^\dagger; \DD_{\lambda})$ is the Pottharst-style Selmer group defined in \S\ref{subsec_SelmerComplexAnalytic} below and $ L_p(f_\lambda, g)$ is the Rankin--Selberg $p$-adic $L$-function associated to the $p$-stabilization $f_\lambda$ of $f$.
\end{lthm}

\subsection*{Forthcoming work} 

 We shall study the special case when $f=g$ with $a_p=0$ in {the subsequent} article \cite{BLV}. In that case, the rank-four module $T$ decomposes into the direct sum of the symmetric square of $R_f^*$ and a rank-one representation, and we use the results in the present paper to study the Iwasawa theory of the symmetric square of $f$ (which is complementary to the work of Loeffler--Zerbes in the ordinary case in \cite{LZ2}). In this special case, we are able to verify the non-vanishing condition that is the key hypothesis in Theorem \ref{thm:liftingBFattrivialtamelevelinFracHINTRO}, allowing us to prove unconditional versions of Theorems \ref{thm:qsignedmainconjINTRO} and \ref{thm:analyticmainconjectureimprovedINTRO} for the symmetric square.
\section{Review on $p$-adic power series}\label{S:review}

We recall the definitions and basic properties of the rings appearing in non-ordinary Iwasawa theory, following \cite[\S2]{LLZCJM}. We fix a finite extension $E / \Qp$ with ring of integers $\cO$, which will be the coefficient field for all the representations we shall consider.

\subsection{Iwasawa algebras and distribution algebras}

Let $\Gamma=\Gal(\QQ(\mu_{p^\infty})/\QQ)$. This group is isomorphic to a direct product $\Delta\times\Gamma_1$, where $\Delta$ is a finite group of order $p-1$ and $\Gamma_1=\Gal(\QQ(\mu_{p^\infty}) / \QQ(\mu_p))$. We choose a topological generator $\gamma$ of $\Gamma_1$, which determines an isomorphism $\Gamma_1 \cong \Zp$.

We write $\Lambda=\cO\lb \Gamma\rb$, the Iwasawa algebra of $\Gamma$. The subalgebra $\cO\lb\Gamma_1\rb$ can be identified with the formal power series ring $\cO\lb X \rb$, via the isomorphism sending $\gamma$ to $1 + X$; this extends to an isomorphism
\begin{equation}\label{eq:lambda}
 \Lambda=\cO[\Delta]\lb  X\rb.
\end{equation}

We may consider $\Lambda$ as a subring of the ring $\cH$ of locally analytic $E$-valued distributions on $\Gamma$. The isomorphism \eqref{eq:lambda} extends to an identification between $\cH$ and the subring of power series $F \in E[\Delta]\lb X \rb$ which converge on the open unit disc $|X| < 1$.

For $n\ge0$, we write $\omega_{n}(X)$ for the polynomial $(1+X)^{p^{n}}-1$. We set $\Phi_0(X) = X$, and $\Phi_n(X) = \omega_n(X)/\omega_{n-1}(X)$ for $n \ge 1$.
We write $\Tw$ for the ring automorphism of $\cH$ defined by $\sigma \mapsto \chi(\sigma)\sigma$ for $\sigma\in \Gamma$.
Let $u=\chi(\gamma)$ be the image of our topological generator $\gamma$ under the cyclotomic character, so that $\Tw$ maps $X$ to $u(1 + X) - 1$. If $m\ge 1$ is an integer, we define
\begin{align*}
 \omega_{n,m}(X)=\prod_{i=0}^{m-1}\Tw^{-i}\left(\omega_{n}(X)\right);
 \Phi_{n,m}(X)=\prod_{i=0}^{m-1}\Tw^{-i}\left(\Phi_{n}(X)\right)
\end{align*}
Let $\log_p$ be the $p$-adic logarithm in $\cH$. We define similarly
\[
\log_{p,m}=\prod_{i=0}^{m-1}\Tw^{-i}\left(\log_p\right).
\]
Finally, we define
\[
\fn_m=\prod_{i=0}^{m-1}\Tw^{-i}\left(\frac{\log_p(1+X)}{X}\right).
\]

\subsection{Power series rings}

Let $\AQp = \cO\lb \pi\rb$, where $\pi$ is a formal variable. We equip this ring with a $\cO$-linear \emph{Frobenius endomorphism} $\varphi$, defined by $\pi\mapsto (1+\pi)^p-1$, and with an $\cO$-linear action of $\Gamma$ defined by $\pi\mapsto(1+\pi)^{\chi(\sigma)}-1$ for $\sigma\in\Gamma$, where $\chi$ denotes the $p$-adic cyclotomic character.

The Frobenius $\varphi$ has a left inverse $\psi$, satisfying
\[ 
(\varphi \circ \psi)(F)(\pi) = \tfrac{1}{p} \sum_{\zeta: \zeta^p = 1} F\left( \zeta(1 + \pi) - 1 \right).
\]
The map $\psi$ is not a morphism of rings, but it is $\cO$-linear, and commutes with the action of $\Gamma$.

We regard $\AQp$ as a subring of the larger ring
\[ \Brig = \left\{F(\pi)\in E\lb \pi\rb: F \text{ converges on the open unit disc}\right\}.\]
The actions of $\vp$, $\psi$, and $\Gamma$ extend to $\Brig$, via the same formulae as before. We shall write $q=\vp(\pi)/\pi\in\AQp$, and $t=\log_p(1+\pi)\in\Brig$. 

\subsection{The Mellin transform}

The action of $\Gamma$ on $1 + \pi \in (\AQp)^{\psi = 0}$ extends to an isomorphism of $\Lambda$-modules
\begin{align*}
 \fM:\Lambda {\stackrel{\cong}{\longrightarrow}}& (\AQp)^{\psi=0}\\
 1\longmapsto& 1+\pi,
\end{align*}
called the \emph{Mellin transform}. This can be further extended to an isomorphism of $\cH$-modules
\[ \cH {\stackrel{\cong}{\longrightarrow}} (\Brig)^{\psi = 0}\]
which we denote by the same symbol.

\begin{theorem}\label{thm:mellin}
 For all integers $m, n \ge 1$, the Mellin transform induces an isomorphism of $\Lambda$-modules
 \[
 \Phi_{n,m}(X)\Lambda\cong \vp^n(q^{m})(\AQp)^{\psi=0}.
 \]
\end{theorem}

\begin{proof}
 See \cite[Theorem~2.1 and equation (2.2)]{LLZCJM}.
\end{proof}

\subsection{Classical and analytic Iwasawa cohomology}
\label{subsec:iwasawacohomology}
Let $\cT$ be a finite-rank free $\cO$-module with a continuous action of $G_F$, where $F$ is a finite unramified extension of $\Qp$. Then the Iwasawa cohomology groups of $\cT$ are classically defined by
\[  H^i_{\Iw}(F(\mu_{p^\infty}), \cT) \coloneqq \varprojlim_n H^i(F(\mu_{p^n}), \cT).\]
{Alternatively}, these can be defined using a version of Shapiro's lemma: set $\TT = \cT \otimes \Lambda^\iota$, where $\Lambda^\iota$ denotes the free rank 1 $\Lambda$-module on which $G_F$ acts via the \emph{inverse} of the canonical character $G_F \twoheadrightarrow \Gamma \hookrightarrow \Lambda^\times$. Then one has
\[ H^i_{\Iw}(F(\mu_{p^\infty}), \cT) \cong H^1(F, \TT). \]
If $F$ is a number field, and $\Sigma$ a finite set of places of $F$ containing all $v \mid p\infty$ and all primes where $\cT$ is ramified, then we can define similarly
\begin{align*}
 H^i_{\Iw, \Sigma}(F(\mu_{p^\infty}), \cT) &\coloneqq \varprojlim_n H^i\left(F_\Sigma / F(\mu_{p^n}), \cT\right) \\& \cong H^i\left(F_\Sigma / F, \TT\right),
\end{align*}
where $F_\Sigma$ is the maximal extension unramified outside $\Sigma$. In both local and global settings, the Iwasawa cohomology groups are finitely-generated as $\Lambda$-modules, and zero unless $i \in \{1, 2\}$. We define Iwasawa cohomology of $\cV = \cT[1/p]$ by tensoring the above groups with $\Qp$.

\begin{remark}
 The group $H^1_{\Iw, \Sigma}(F(\mu_{p^\infty}), \cT)$ is actually independent of the choice of $\Sigma$, and we will frequently drop $\Sigma$ from the notation. This is not the case for $H^2_{\Iw, \Sigma}$.
\end{remark}

The ``analytic'' variants of these modules, which play a key role in Pottharst's approach to cyclotomic Iwasawa theory of non-ordinary motives \cite{jaycyclo, jayanalyticselmer}, are obtained by systematically replacing $\Lambda$ with the larger ring $\cH$. We define $\VV^\dag = \cV \otimes \cH^\iota$; then for $F$ a $p$-adic field we have
\begin{align*}
 H^i_{\an}(F(\mu_{p^\infty}), \cV) &\coloneqq H^i(F, \VV^\dag) \\
 &\cong \cH \otimes_{\Lambda[1/p]} H^i_{\Iw}(F(\mu_{p^\infty}), \cV),
\end{align*}
and similarly for the global setting. The importance of the analytic Iwasawa cohomology groups is that for $F$ a $p$-adic field, the analytic Iwasawa cohomology of $\cV$ is encoded in its Robba-ring $(\varphi, \Gamma)$-module; see \S \ref{sect:phigammacoho} below. 

\subsection{The Perrin-Riou regulator map}

Let $F$ be an unramified extension of $\Qp$, and $\cT$ an $\cO$-representation of $G_F$ as before; and assume that $\cV = \cT[1/p]$ is crystalline, with all Hodge--Tate weights\footnote{Our convention is that the Hodge--Tate weight of the cyclotomic character is $+1$.} $\ge 0$, and with no quotient isomorphic to the trivial representation. We also fix a choice of $p$-power roots of unity $\zeta_{p^n} \in \overline{\QQ}_p$, for $n \ge 1$. 

Then there is a canonical homomorphism of $\cH$-modules, the \emph{Perrin-Riou regulator},
\[ \cL_{F, \cV}: H^1_{\Iw}\left(F(\mu_{p^\infty}), \cV\right) \to \cH \otimes \Dcris(F, \cV) \]
which interpolates the values of the Bloch--Kato logarithm and dual-exponential maps for twists of $\cV$ by locally algebraic characters of $\Gamma$. 

It will be important to us later to consider how these maps interact with change of the field $F$. If $K / F$ is an unramified extension with Galois group $U$, then $\Dcris(K, \cV) = K \otimes_F \Dcris(F, \cV)$; so the source and target of the regulator map $\cL_{K, \cV}$ are naturally modules over the larger group $K(\mu_{p^\infty}) / F \cong \Gamma \times U$, and it follows easily from the construction that $\cL_{K, \cV}$ commutes with the action of this group.

Moreover, we have an interaction with restriction and corestriction maps which can be summarized by the following diagram:
\begin{equation}
 \begin{tikzcd}[row sep=huge, column sep=huge]
  H^1_{\Iw}(K(\mu_{p^\infty}), \cV) \ar[r, "\cL_{K, \cV}"] \ar[d, "\operatorname{cores}"', bend right] & \cH \otimes \Dcris(K, \cV) \ar[d, "\operatorname{trace}"', bend right] \\
  H^1_{\Iw}(F(\mu_{p^\infty}), \cV) \ar[u, "\operatorname{res}"', bend right]\ar[r, "\cL_{F, \cV}"] & \cH \otimes \Dcris(F, \cV)\ar[u, "\subseteq"', bend right].
 \end{tikzcd}
\end{equation}

\section{Euler systems of rank $2$ for Rankin--Selberg products (Conjectures)}

We expect that the Beilinson--Flach Euler system may be obtained by applying a suitable ``rank-lowering operator'' to a rank-2 Euler system. Our goal in this section is to present a precise account of this expectation and formulate a conjecture. Even though we are currently unable to verify this conjecture in general, it serves as a sign-post for the signed-splitting procedure for the $p$-stabilized Beilinson--Flach elements that we will develop in the later sections.

\subsection{Review of Perrin-Riou's theory}

Let $\cT$ be a free $\cO$-module of finite rank with a continuous action of the absolute Galois group $G_{\mathbb{Q}}$, which is unramified outside a finite set of primes $\Sigma \ni p$. Let $\cV = \cT \otimes_{\cO} E$.

Let $\mathcal{P}$ denote a set of primes disjoint from $\Sigma$, and let $\NP$ denote the set of square-free integers whose prime divisors are in $\mathcal{P}$. For an integer $m \in \NP$, we set $\Delta_m=\Gal(\QQ(\mu_m)/\QQ)$ and $\LL_m\coloneqq \cO[[\Gal(\QQ(\mu_{mp^\infty}) / \QQ)]] \cong \LL\otimes_{\ZZ_p}\ZZ_p[\Delta_m]$.

\begin{defn} An Euler system of rank $r \geq 0$ is a collection of classes 
 \[ c_{m} \in \bigwedge_{\LL_m}^{r} \HIw(\QQ(\mu_{mp^{\infty}}), \cT) \]
 for each $m \in \NP$, such that if $\ell$ is a prime with $\ell, m\ell \in \NP$, then
 \begin{equation}
  \label{eq:ESrelation}
  \cor_{\QQ(\mu_{m\ell p^{\infty}})/\QQ(\mu_{mp^{\infty}})}\left(c_{m\ell}\right) =   
  \begin{cases}
   P_{\ell}(\sigma_{\ell}^{-1})c_{m}& \text{if } \ell \nmid m,\\
   c_{m}                            & \text{if } \ell \mid m.
  \end{cases}
 \end{equation}
\end{defn}

Here $P_{\ell}(X) \coloneqq \text{det}_E(1 - \text{Frob}_{\ell}^{-1}X \mid V^{*}(1)) \in \cO[X]$, and $\sigma_{\ell}$ denotes the image of $\operatorname{Frob}_\ell$ in $\Gal(\QQ(\mu_{mp^\infty}) / \QQ)$, where $\operatorname{Frob}_\ell$ is the arithmetic Frobenius at $\ell$.

\begin{remark}
 \label{rmk:rank0}
 Perrin-Riou in fact requires $r \ge 1$, but we feel that the case $r = 0$ should not be neglected. For $r = 0$ we have $\bigwedge_{\LL_m}^{r} \HIw(\QQ(\mu_{mp^{\infty}}), \cT) = \LL_m$, and a rank 0 Euler system is therefore a collection of elements $c_m \in \Lambda_m$ for $m \in \NP$, satisfying the compatibilities \eqref{eq:ESrelation} under the projection maps $\Lambda_{m\ell} \to \Lambda_m$. Such collections of elements arise naturally in the theory of $p$-adic $L$-functions: for instance, both the Stickelberger elements for an odd Dirichlet character, and the Mazur--Tate elements for a $p$-ordinary modular form, can be viewed as rank 0 Euler systems in this sense.
\end{remark}  

Given an Euler system of rank $r > 1$, one can construct a multitude of Euler systems of rank $1$ with the aid of auxiliary choices of functionals on the Iwasawa cohomology, following a recipe originally set out by Rubin in \cite{Rub96} and later formalized by Perrin-Riou in \cite{PR98}. 

\begin{defn}
 A \emph{Perrin-Riou functional} of rank $s \ge 1$ is a collection of linear functionals $\{ \Phi_m : m \in \NP \}$, where
 \[ \Phi_m \in \bigwedge^{s}_{\Lambda_m} \Hom_{\Lambda_m} \Big(\HIw(\QQ(\mu_{mp^{\infty}}),\cT), \Lambda_m\Big), \]
 such that if $\ell$ is a prime with $\ell, m\ell \in \NP$, we have
 \begin{equation}
  \label{eq:prcompat} 
  \Phi_{m\ell} \circ \res_{\QQ(\mu_{m\ell p^{\infty}})/\QQ(\mu_{mp^{\infty}})} = \iota_{m\ell / m} \circ \Phi_m,
 \end{equation}
 where $\iota_{ml/m}$ denotes the isomorphism $\Lambda_m \cong \left(\Lambda_{m\ell}\right)^{\Delta_\ell}$ sending $1$ to $\sum_{\sigma \in \Delta_\ell} [\sigma]$.
\end{defn}
As in \cite[Corollary 1.3]{Rub96}, one may interpret a rank $r-1$ Perrin-Riou functional $\Phi= \{\Phi_m\}$ as a collection of maps
\[ \Phi_m : \bigwedge^{r}_{\Lambda_m} \HIw(\QQ(\mu_{mp^{\infty}}), \cT) \rightarrow  H^1_{\text{Iw}}(\QQ(\mu_{mp^{\infty}}), \cT) .\]
\begin{proposition}[{Perrin-Riou}] 
 
 If $\lbrace c_m \rbrace_{m \in N(\mathcal{P})}$ is an Euler system of rank $r$  and $\{\Phi_m\}$ is a Perrin-Riou functional of rank $r -1$, then
 \[ \Phi_m(c_m) \in \HIw(\QQ(\mu_{mp^{\infty}}), \cT) \] 
 is a rank one Euler system. 
\end{proposition}

\begin{proof} See \cite[Lemma 1.2.3]{PR98} and  \cite[\S 6]{Rub96}.
\end{proof}

\begin{remark}
 More generally, one may interpret a rank $s$ Perrin-Riou functional as a ``rank-lowering operator'' sending rank $r$ Euler systems to rank $r - s$ Euler systems. This includes the case $r = s$, where we understand rank 0 Euler systems as in Remark \ref{rmk:rank0} above.
\end{remark}

\subsection{Analytic Euler systems}

For the applications below, we will need to consider compatible families of classes {not} lying in Iwasawa cohomology, but in the larger ``analytic'' cohomology modules of Pottharst. For simplicity, we shall only describe this construction in rank 1.

We recall that $\cH$ can be written as an inverse limit $\varprojlim_n \cH[n]$, where $\cH[n]$ are reduced affinoid algebras, and for each $n$ we have
\[ \cH[n] \otimes_{\cH} H^1_{\an}(\QQ(\mu_{mp^\infty}), \cV) = H^1(\QQ(\mu_m), \cH_n \otimes_E \cV) \]
by \cite[Theorem 1.7]{jayanalyticselmer}. Each $\cH[n]$ has a canonical supremum norm; if $\cH[n]^\circ$ denotes the unit ball for this norm, then there is a seminorm $\|\cdot\|_n$ on $H^1(\QQ(\mu_m), \cH[n] \otimes_E \cV)$ for which the unit ball is the image of $H^1(\QQ(\mu_m), \cH[n]^\circ \otimes_\cO \cT)$. (In fact this is a norm, by \cite[Proposition 2.1.2(1)]{LZ1}, but we do not need this.)

\begin{defn}
 \label{def:anES}
 An \emph{analytic Euler system} (of rank 1) for $\cV$ is a collection of classes $c_m \in H^1_{\an}(\QQ(\mu_{mp^\infty}), \cV)$, for each $m \in \NP$, satisfying the following two conditions:
 \begin{enumerate}
  \item if $\ell$ is prime and $m, m\ell \in \NP$, then the norm-compatibility condition \eqref{eq:ESrelation} holds;
  \item for each $n$, there is a constant $C_n$ (independent of $m$) such that $\|c_m\|_n \le C_n$ for all $m \in \NP$.
 \end{enumerate}
\end{defn}

\begin{remark}
 Condition (1), asserting that there is no ``growth in the tame direction'', is technical to state but absolutely vital in order to obtain an interesting theory; it is trivial that any class in $H^1_{\an}(\QQ(\mu_{p^\infty}), \cV)$ can be extended to a compatible family of classes satisfying (2) alone.
\end{remark}

\subsection{Unbounded Perrin-Riou functionals}

In this section, building on \cite{OTS09} and \cite{LLZ1}, we will construct \emph{canonical} Perrin-Riou functionals using another construction of Perrin-Riou, namely the $p$-adic regulator map. The price we pay for this canonicity is that our functionals are no longer bounded in general. 

We now assume $\cV$ is crystalline, with all Hodge--Tate weights $\ge 0$, and that $\cV |_{G_{\Qp}}$ has no quotient isomorphic to the trivial representation. We have already chosen a compatible family of $p$-power roots of unity $\zeta_{p^r}$. For $\mathcal{P}$ as above, let us also choose a primitive $n$-th root of unity $\zeta_\ell$ for each $\ell \in \mathcal{P}$. One checks easily that if $m \in \NP$, then $\xi_m = \prod_{\ell \mid m} (-\zeta_\ell)$ is a basis vector of the ring of integers $\ZZ[\mu_m]$ as a free rank 1 module over the group ring $\ZZ[\Delta_m]$; and we have the trace-compatibility
\[ \operatorname{trace}_{m\ell / m}(\xi_{m\ell}) = \xi_m.\]

\begin{defn}
 For $m \in \NP$, let $\nu_m: \ZZ[\zeta_m] \to \ZZ[\Delta_m]$ denote the unique $\ZZ[\Delta_m]$-linear map sending $\xi_m$ to $1$.
\end{defn}

Let us now set $\cH_m = \Zp[\Delta_m] \otimes \cH$, which we regard as an ``analytification'' of $\Lambda_m$. For $\cV$ as above, the sum of the Perrin-Riou regulators at the primes of $\QQ(\mu_m)$ above $p$ gives a map
\[
H^1_{\Iw}(\QQ(\mu_{mp^\infty}) \otimes \Qp, \cV) \lra \QQ(\mu_m) \otimes_{\QQ} \cH \otimes_E \Dcris(\Qp, \cV),
\]
and composing this with $\nu_m$ we obtain a morphism of $\cH_m$-modules
\begin{equation}
 \label{eq:equivariantreg}
 \cL_{m, \cV}: H^1_{\Iw}(\QQ(\mu_{mp^\infty}) \otimes \Qp, \cV) 
 \lra \cH_m \otimes \Dcris(\Qp, \cV).
\end{equation} 
(We shall abbreviate $\cL_{1, \cV}$ by $\cL_V$.)

\begin{defn}
 If $t \in \Dcris(\Qp, \cV^*(1))$, and $m \in \NP$, then we define a map 
 \[ \Phi^{(t)}_m : H^1_{\Iw}(\QQ(\mu_{mp^\infty}), \cV) \to \cH_m,\qquad \Phi^{(t)}_m(z)\coloneqq \left\langle \cL_{m, \cV}(\operatorname{loc}_p z), t \right\rangle. \]
\end{defn} 

One sees easily that, for any fixed $t$, the collection $\Phi^{(t)} = \{ \Phi^{(t)}_m : m \in \NP \}$ satisfies the compatibility condition \eqref{eq:prcompat}, and thus may be regarded as a (rank 1) \emph{unbounded Perrin-Riou functional}. Pairing with $\Phi^{(t)}$ therefore defines a homomorphism from Euler systems of rank 2 to (possibly unbounded) analytic Euler systems of rank 1.

\begin{remark}
 Via exactly the same construction, for any $s \ge 1$ we may use elements of the wedge power $\bigwedge^s_E \Dcris(\Qp, \cV^*(1))$ to define unbounded Perrin-Riou functionals of rank $s$.
\end{remark}

\subsection{The Beilinson--Flach Euler systems} \label{sect:BFelt}

We now focus on the particular case which interests us: the Rankin--Selberg convolution of two modular forms.
As in the introduction,  $f$ and $g$ denote two normalised, new cuspidal modular eigenforms, of weights $k_f + 2, k_g + 2$, levels $N_f, N_g$, and characters $\epsilon_f$, $\epsilon_g$ respectively. We assume that $p \nmid N_f N_g$. We let $T$ denote the rank 4 representation $R_f^* \otimes R_g^*$ over $\cO$ and write $V=T\otimes_{\Zp} \Qp$. 

We take $\Sigma$ to be the set of primes dividing $p N_f N_g$, and for $\ell \notin \Sigma$, we write
\[ 
P_\ell(X) = \det\left(1 - \operatorname{Frob}_\ell^{-1} X\, \middle|\,  T^*(1)\right) = 1 - \frac{a_\ell(f) a_\ell(g)}{\ell} X + \dots \in \cO[X].
\]

We now review the construction of Beilinson--Flach elements from \cite{LZ1}. For $\lambda\in\{\alpha,\beta\}$ and $h\in\{f,g\}$, we write $h^\lambda$ for the $p$-stabilisation of $h$ at $\lambda_h$. We shall identify $R_h^*$ with $R_{h^{\lambda}}^*$ following  \S3.5 of \textit{op. cit.} More specifically, let $\pr_1$ and $\pr_2$ be the two degeneracy maps on the modular curves $Y_1(pN_h)\rightarrow Y_1(N_h)$ as defined in \cite[Definition~2.4.1]{KLZ2} and write  ${\Pr}^{\lambda}=\pr_1-\frac{\lambda'}{p^{k_h+1}}\pr_2$, where $\lambda'$ denotes the unique element of $\{\alpha,\beta\}\setminus\{\lambda\}$.  Realizing $R_{h^\lambda}^*$ and $R_h^*$ as quotients of the \'etale cohomology of $Y_1(pN_h)$ and $Y_1(N_h)$ respectively, $\Pr^{\lambda}_*$ gives an isomorphism between these two Galois representations.

\begin{defn}
 For $\lambda,\mu\in\{\alpha,\beta\}$, $c > 1$ coprime to $6pN_f N_g$, $m \ge 1$ coprime to $pc$, and $a\in (\ZZ / mp^\infty \ZZ)^\times$, let
 \[
 {}_c\cBF^{\lambda,\mu}_{m, a} \in
 D_{\ord_p(\lambda_f\mu_g)}(\Gamma) \otimes_{\Lambda} \HIw(\QQ(\mu_{mp^\infty}), T)
 \]
 be the Beilinson--Flach element as constructed in \cite[Theorem 5.4.2]{LZ1}.
\end{defn}

Here $D_{\ord_p(\lambda_f\mu_g)}(\Gamma, E)$ denotes the $\Lambda$-submodule of $\cH$ consisting of tempered distributions of order $\ord_p(\lambda_f\mu_g)$. We shall take $a = 1$ throughout, and restrict to integers $m \in \NP$, where $\mathcal{P}$ is the set of primes not dividing $pcN_f N_g$. 

\begin{remark}
 If $\epsilon_f \epsilon_g$ is non-trivial, then we may remove the dependence on the auxiliary integer $c$, but this will not greatly concern us here: we shall simply fix a value of $c$ and drop it from the notation.
\end{remark}

These elements satisfy a norm-compatibility relation which is close, but not identical, to Equation \eqref{eq:ESrelation}. As explained in \cite[Lemma 7.3.2]{LLZ1}, we can modify these elements to ``correct'' the norm relation: there exists a collection of elements $\BF_{\lambda, \mu, m}$ for $m \in \NP$ such that
\begin{itemize}
 \item the $\BF_{\lambda, \mu, m}$ for {$m$} varying satisfy Equation \eqref{eq:ESrelation} exactly,
 \item $\BF_{\lambda, \mu, 1} = {}_c\cBF^{\lambda,\mu}_{1, 1}$,
 \item each $\BF_{\lambda, \mu, m}$ is an $\cO[\Delta_m]$-linear combination of the elements ${}_c\cBF^{\lambda,\mu}_{m', 1}$ for $m' \mid m$.
\end{itemize}

As in \cite[Theorem 8.1.4(ii)]{LZ1}, if $H^0(\QQ(\mu_{p^\infty}), V) = 0$, the collection of elements $\BF_{\lambda, \mu, m}$ for varying $m \in \NP$ form an analytic Euler system in the sense of Definition \ref{def:anES}. We thus obtain four rank 1 analytic Euler systems for $T$, one for each of the possible choices of $p$-stabilisations $\lambda, \mu$.

One of the key themes in the present paper will be to understand the relations among these Euler systems, for a fixed $f$ and $g$ and different choices of $p$-stabilisations. Our first result in this direction is the following straightforward compatibility. For $\chi$ any continuous character of $\Gamma$, and $z \in \cH \otimes H^1_{\Iw}(\QQ(\mu_{mp^\infty}), V)$, let us write $z(\chi)$ for the image of $z$ in $H^1(\QQ(\mu_m), V(\chi^{-1}))$.

\begin{lemma}
 \label{lemma:compareBF}
 Let $m \in \NP$, and let $\chi$ be a character of $\Gamma$ of the form $z \mapsto z^j \theta(z)$, where $j \in \ZZ$ and $\theta$ is a finite-order character of conductor $p^r$.
 \begin{enumerate}
  \item[(i)] If $0 \le j \le \min(k_f, k_g)$, then
  \[
  (\lambda_f\mu_g)^r \cdot \BF_{\lambda, \mu, m}(\chi)
  \]
  is independent of the choice of $\lambda$ and $\mu$.
  \item[(ii)] If $k_g < j \le k_f$ then this class is independent of $\lambda$ (but may depend on $\mu$); and similarly if $k_f < j \le k_g$ it is independent of $\mu$.
 \end{enumerate}
 If $\chi$ is the character $z \mapsto z^j$, the same conclusions hold for the class
 \[\left(1 - \tfrac{\lambda \mu}{p^{1 + j} \sigma_p}\right)\left(1 - \tfrac{p^j \sigma_p}{\lambda \mu}\right)^{-1} \BF_{\lambda, \mu, m}(\chi).\]
\end{lemma}

\begin{proof}
 Part (i) follows from the same proof as {\cite[Proposition~3.3 and Corollary~3.4]{BL16b} since the infinite part of $\chi$ (the character $z\mapsto z^j$) has the effect of sending the Beilinson--Flach element for the representation $T$ to that for the Tate twist $T(-j)$}.  For part (ii), we assume $k_g < j \le k_f$ without loss of generality, and deform $g_\mu$ in a Coleman family $\cG$ (while keeping $f$ and $\theta$ fixed). By Theorem A of \cite{LZ1}, we obtain two families of cohomology classes $\BF_{\alpha, \cG, m}(\chi)$ and $\BF_{\beta, \cG, m}(\chi)$, and by part (i) the specialisations of these at integer points $r' \ge j$ are equal. Hence the two families of classes are equal identically, and we obtain (ii) by specialising back to $g_\mu$. 
\end{proof}

\subsection{A conjectural rank 2 Euler system}\label{S:conjrank2ES}

Recall that $T$ is a free $\cO$-module of rank $4$ and observe that both $-1$ and $+1$-eigenspaces for the action of complex conjugation on $T$ have rank $2$. In this situation, we expect to have an Euler system of rank $r=2$.

We are now ready to state our conjecture on the relation of $p$-stabilized Beilinson--Flach classes and rank-2 Euler systems. 
Let $\cL_{m, V}: H^1_{\Iw}\left(\mathbb{Q}(\mu_{mp^\infty}), V\right) \to \cH_m \otimes \Dcris(V)$ be the equivariant Perrin-Riou regulator as in Equation \eqref{eq:equivariantreg}, and let $\{v_{\lambda\mu}\}_{\lambda,\mu\in\{\alpha,\beta\}}$ be an eigenvector basis of $\Dcris(V)$ in which the matrix of $\varphi$ is given by
\[ D \coloneqq \begin{pmatrix}
\frac{1}{\alpha_f\alpha_g}\\
&\frac{1}{\alpha_f\beta_g}\\
&&\frac{1}{\beta_f\alpha_g}\\
&&&\frac{1}{\beta_f\beta_g}
\end{pmatrix}.\]  
Further, let $\{v_{\lambda\mu}^*\}$ be the dual basis to $\{v_{\lambda\mu}\}$.

These vectors are \emph{a priori} only determined up to scaling; we may normalise them canonically as follows. The 1-dimensional space $\frac{\Dcris(V_f)}{\Fil^1}$ has a canonical basis vector $\eta_f'$, as defined in \cite[\S 6.1]{KLZ1a}. Since we are assuming $f$ to be non-ordinary, the two eigenspaces are both complementary to $\Fil^1$, and we can thus define $v_{f, \alpha}$ and $v_{f, \beta}$ to be the unique vectors in the $\varphi$-eigenspaces satisfying
\[ v_{f, \alpha} = v_{f, \beta} = \eta_f' \bmod \Fil^1. \]
Defining $v_{g, \mu}$ analogously, we can choose our eigenvector basis of $\Dcris(V^*) = \Dcris(V_f) \otimes \Dcris(V_g)$ by setting $v_{\lambda \mu}^* = v_{f, \lambda} \otimes v_{g, \mu}$.

\begin{conjecture}
 \label{conj:rank2} 
 There exists a collection of classes $\BF_m\in \bigwedge^2H^1_\Iw(\QQ(\mu_m),T)$, for all $m \in \NP$, which form a rank 2 Euler system, and are such that for all $\lambda,\mu\in\{\alpha,\beta\}$, we have
 \[
 \left\langle \cL_{m, V}(\BF_m), v_{\lambda, \mu}^* \right\rangle = \BF_{\lambda, \mu, m}.
 \]
\end{conjecture}

Equivalently, the four rank 1 analytic Euler systems $\left(\BF_{\lambda, \mu, m}\right)_{m \in \NP}$, for different choices of $\lambda$ and $\mu$, are all obtained from the single rank 2 Euler system $(\BF_m)$ via the Perrin-Riou functionals associated to the four eigenvectors $v_{\lambda, \mu}^*$. 

\begin{remark}
 This conjecture is an extension to higher-weight modular forms of the conjectures formulated in \cite[\S 8]{LLZ1} for pairs of weight 2 modular forms. At the time, this conjecture was somewhat tentative since the methods of \emph{op.cit.} only suffice to construct the classes $\BF_{\lambda, \mu, m}$ when $v_p(\lambda\mu) < 1$, which is satisfied for at most two of the four possible choices, and sometimes for none at all. However, this restriction has since been removed in \cite{LZ1} via the use of Coleman families.
\end{remark}

\subsection{$p$-adic $L$-functions and explicit reciprocity laws}
\label{geometricp-adiclfunction}

For $\lambda \in \{\alpha_f, \beta_f\}$ and $\mu \in \{\alpha_g, \beta_g\}$, there exist Coleman families $\cF$ and $\cG$ passing through the $p$-stabilisations $f_\lambda$ and $g_\mu$; these are families of overconvergent eigenforms over some affinoid discs $V_{1}$ and $V_{2}$ in the weight space $\mathcal{W}$. We suppose (temporarily) that our coefficient field $E$ contains a primitive $N$-th root of unity, where $N = \operatorname{LCM}(N_f, N_g)$.

\begin{theorem}[Loeffler--Zerbes, \cite{LZ1}]
 \label{thm:LZgeominterpolationformula}
 There exists a 3-variable $p$-adic $L$-function $L_p^{\mathrm{geom}}(\cF, \cG) \in \cO(V_1 \times V_{2} \times \mathcal{W})$ with the following interpolation property. Let $(r, r^\prime, j)$ be an integer point in $V_1 \times V_{2} \times \mathcal{W}$ such that $r \geq 0, r^\prime \geq -1$ and $ \frac{r+ r^\prime + 1}{2} \leq j \leq r$. Suppose that the specializations $\cF_r$ and $\cG_{r^\prime}$ are $p$-stabilizations of classical newforms $f_r$ and $g_{r^\prime}$ of prime-to-$p$ level. Then,
 \begin{align*} L_p^{\mathrm{geom}} (\cF, \cG) (r, r^\prime, j) = \frac{\mathcal{E}(f_r, g_{r^\prime}, 1+j)}{\mathcal{E}(f_r)\mathcal{E}^{*}(f_r)} &\times \frac{j!(j - r^\prime -1)!(c^2 - c^{2j-r-r'} \epsilon_\cF(c)^{-1} \epsilon_\cG(c)^{-1})}{\pi^{2j - r^\prime + 1}(-1)^{r - r^\prime}2^{2j + 2 + r - r^\prime}}\\
  &\times \frac{L(f_r, g_{r^\prime}, 1 + j)}{\langle f_r, f_r \rangle_{N_f}}
 \end{align*}
 where
 \[ \mathcal{E}(f_r) = \Bigg( 1 - \frac{\lambda^\prime_r}{p\lambda_r} \Bigg) ,\qquad\qquad \mathcal{E}^{*}(f_r) = \Bigg( 1 - \frac{\lambda^\prime_r}{\lambda_r} \Bigg),   \]
 \[ \mathcal{E}(f_r, g_r', 1+j) = \Bigg( 1 - \frac{p^j}{\lambda_r\mu_r} \Bigg) \Bigg( 1 - \frac{p^j}{\lambda_r\mu^\prime_r} \Bigg) \Bigg( 1 - \frac{\lambda^\prime_r\mu_r}{p^{1+j}} \Bigg) \Bigg( 1 - \frac{\lambda^\prime_r\mu^\prime_r}{p^{1+j}} \Bigg).   \]
 Here, $\lambda_r,\mu_{r^\prime}$ are the respective specializations of the $U_p$-eigenvalues on $\cF$ and $\cG$ at $r$ and $r^\prime$, whereas $\lambda^\prime_r$ and $\mu^\prime_r$ are defined by the requirement that $\{\lambda_r,\lambda^\prime_r\}=\{\alpha_{f_r},\beta_{f_r}\}$ and $\{\mu,\mu^\prime\}=\{\alpha_{g_{r^\prime}},\beta_{g_{r^\prime}}\}$.
\end{theorem}

\begin{remark}
 The construction of this function in \cite{LZ1} relies on deforming Beilinson--Flach elements in families. An alternative, more direct construction (not using Euler systems) has subsequently been given by Urban in \cite[Appendix II]{AIU}.
\end{remark}

\begin{proposition}\label{prop:samegeom}
 Let $L_p(\cF, g_\mu)$ denote the function on $V_1 \times \mathcal{W}$ obtained by specialising $L_p^{\mathrm{geom}}(\cF, \cG)$ at the point of $V_2$ corresponding to $g_\mu$. Then the functions $L_p(\cF, g_\alpha)$ and $L_p(\cF, g_\beta)$ coincide.
\end{proposition}

\begin{proof}
 From the preceding theorem, one sees that these two functions agree at all points $(r, j)$ with $r$, $j$ integers satisfying the inequalities $r \ge 0$ and $\tfrac{r + k_g + 1}{2} \le j \le r$. These points are clearly Zariski-dense in $V_1 \times \mathcal{W}$.
\end{proof}

\begin{defn}
 For $\lambda \in \{\alpha_f, \beta_f\}$, define $L_p(f_\lambda, g) \in \cO(\mathcal{W})$ to be the specialisation of
 \[ \left[w(f) G(\epsilon_f^{-1}) G(\epsilon_g^{-1}) \mathcal{E}(f) \mathcal{E}^*(f)\right] \cdot L_p(\cF, g)\]
 at the point $f_\lambda$ of $V_1$, where $L_p(\cF, g)$ is the common value $L_p(\cF, g_\alpha) = L_p(\cF, g_\beta)$, $G(\dots)$ are the Gauss sums, and $w(f)$ is the Atkin--Lehner pseudo-eigenvalue of $f$.
\end{defn}

One can check that $L_p(f_\lambda, g)$ is defined over any $p$-adic field containing the coefficients of $f_\alpha$ and $g$ (not necessarily containing an $N$-th root of unity); this is the reason for renormalising by the Gauss sums.  Since there is a canonical isomorphism $\cO(\mathcal{W}) \cong \cH$, we shall regard $L_p(f_\lambda, g)$ as an element of $\cH$. We therefore have four $p$-adic $L$-functions attached to the pair $\{ f, g\}$, namely
\begin{equation}
 \label{eq:fourLp} 
 \Big\{ L_p(f_\alpha, g),\quad L_p(f_\alpha, g), 
 \quad  L_p(g_\alpha, f),\quad L_p(g_\beta, f)\Big\}.
\end{equation}

\begin{theorem}[Explicit reciprocity law]
 \label{thm:explicitrecip}
 For each pair $(\lambda, \mu)$ we have
 \begin{align*}
  \langle \cL_V(\BF_{\lambda, \mu, 1}), v_{\lambda, \mu}^* \rangle &= 0,\\
  \langle \cL_V(\BF_{\lambda, \mu, 1}), v_{\lambda, \mu'}^* \rangle &= \frac{A_g \log_{p, 1+k_g}}{(\mu' - \mu)} \cdot L_p(f_\lambda, g),\\
  \langle \cL_V(\BF_{\lambda, \mu, 1}), v_{\lambda', \mu}^* \rangle &= \frac{A_f \log_{p, 1+k_f}}{(\lambda' - \lambda)} \cdot L_p(g_\mu, f)
 \end{align*}
 where $A_f$ and $A_g$ are non-zero constants independent of $\lambda$ and $\mu$. In particular we have the anti-symmetry relations
 \begin{align*} 
  \langle \cL_V(\BF_{\lambda, \mu, 1}), v_{\lambda, \mu'}^* \rangle &= -\langle \cL_V(\BF_{\lambda, \mu', 1}), v_{\lambda, \mu}^* \rangle,\\
  \langle \cL_V(\BF_{\lambda, \mu, 1}), v_{\lambda', \mu}^* \rangle &= -\langle \cL_V(\BF_{\lambda', \mu, 1}), v_{\lambda, \mu}^* \rangle.
 \end{align*}
\end{theorem}

\begin{proof}
 The vanishing of $\langle \cL_V(\BF_{\lambda, \mu}), v_{\lambda, \mu}^* \rangle$ is a consequence of Theorem 7.1.2 of \cite{LZ1}. The other two formulae follow directly from the definition of the geometric $p$-adic $L$-function (Definition 9.1.1 of \emph{op.cit.}) after a somewhat tedious comparison of conventions. The factor $\log_{p, 1 + k_g}$ arises from the normalisation of the Perrin-Riou regulator for a certain subquotient of the $(\varphi, \Gamma)$-module of $V$ (cf.~Theorem 7.1.4 of \emph{op.cit}). The quantity $\frac{A_g}{\mu'-\mu}$ and its cousin arise from comparing the families of eigenvectors constructed there with our present conventions; some handle-turning shows that the specialisation of the family $\eta_{\cF}$ in their notation corresponds to $\frac{1}{w(f) G(\epsilon_f^{-1}) \mathcal{E}(f_\lambda) \mathcal{E}^*(f_\lambda)} v_{f, \alpha}$, while the family $\omega_{\cG}$ specialises to $\frac{(-1)^{k_g}\langle \varphi(\omega_g'), \omega'_{g^*} \rangle}{(\mu' - \mu) N_{\epsilon_g} G(\epsilon_g^{-1})} v_{g, \beta}$, where $\omega_g'$ is the basis vector of $\Fil^1 \Dcris(R_g)$ defined in \cite[\S 6.1]{KLZ1a}, and $\omega'_{g^*}$ its analogue for the conjugate form $g^*$.
\end{proof}

\subsection{Some ``extra'' $p$-adic $L$-functions}
\label{sect:extraL}
\begin{defn}
 For $\lambda \in \{\alpha_f, \beta_f\}$ and $\mu \in \{\alpha_g, \beta_g\}$, we set
 \[ L_p^{?}(f_\lambda, g_\mu) \coloneqq \frac{1}{\log_{p, \nu+1}} \left\langle \cL_V(\BF_{\lambda, \mu, 1}), v_{\lambda', \mu'}^*\right\rangle,
 \]
 where $\nu = \min(k_f, k_g)$.
\end{defn}

These elements lie in $\cH$, because $\BF_{\lambda, \mu, 1}(\chi)$ is in $H^1_\mathrm{f}$ for every locally-algebraic $\chi$ of weight in the range $[0, \dots, \nu]$, so $\cL_V(\BF_{\lambda, \mu, 1})$ vanishes at these characters and thus is divisible by $\log_{p, \nu + 1}$.

\begin{proposition}
 Suppose $k_f > k_g$, and let $\chi$ be a character of $\Gamma$ of the form $z \mapsto z^j \theta(z)$, where $k_g + 1 \le j \le k_f$ and $\theta$ is a Dirichlet character of conductor $p^n$. Then we have
 \[ L_p^{?}(f_\lambda, g_\mu)(\chi) = R \cdot \frac{A_g}{\mu' - \mu}\cdot L_p(f_{\lambda'}, g)(\chi),\]
 where $A_g$ is as in the statement of Theorem~\ref{thm:explicitrecip}$;$ $R = \left( \frac{\lambda'}{\lambda} \right)^n$ if $n \ge 1$, and if $n = 0$ then
 \[ R = \frac{ \left(1 - \tfrac{\lambda' \mu}{p^{1 + j} \sigma_p}\right)\left(1 - \tfrac{p^j \sigma_p}{\lambda' \mu}\right)^{-1}}{\left(1 - \tfrac{\lambda \mu}{p^{1 + j} \sigma_p}\right)\left(1 - \tfrac{p^j \sigma_p}{\lambda \mu}\right)^{-1}}.\]
\end{proposition}

\begin{proof}
 Applying $\langle \cL_V(-), v_{\lambda',\mu'}^*\rangle(\chi)$ to both $\BF_{\lambda,\mu,1}$ and $\BF_{\lambda^\prime,\mu,1}$, we deduce that
 \begin{align*}
  L_p^{?}(f_\lambda, g_\mu)(\chi)&= \frac{\langle \cL_V(\BF_{\lambda,\mu,1}), v_{\lambda',\mu'}^*\rangle(\chi)}{\log_{p, \nu+1}}\\
  &=R\cdot\frac{\langle \cL_V(\BF_{\lambda^\prime,\mu,1}), v_{\lambda',\mu'}^*\rangle(\chi)}{\log_{p, \nu+1}}\\
  &=R\cdot\frac{A_g}{\mu^\prime-\mu}L_p(f_{\lambda^\prime},g)(\chi)
 \end{align*}
 where the second equality follows from  Lemma~\ref{lemma:compareBF}, the final equality from Theorem~\ref{thm:explicitrecip} and the first from definitions.
\end{proof}

Note that for $n = 0$ the right-hand side is an explicit multiple of a complex $L$-value, by the explicit reciprocity law (and we expect this also to hold for $n \ge 1$). This construction gives rise to four extra elements associated to $f$ and $g$, in addition to the more familiar four given by \eqref{eq:fourLp}. It seems natural to conjecture that
\begin{equation}
 \label{eq:extraLsym}
 L_p^{?}(f_\lambda, g_\mu) = -L_p^{?}(f_{\lambda'}, g_{\mu'}),
\end{equation}
so that these four extra elements fall into two pairs differing by signs; this would, for instance, follow easily from Conjecture \ref{conj:rank2}. However, we do not know how to prove this symmetry property unconditionally, since these elements do not seem to deform in Coleman families, and their growth (which is always $O(\log_{p, 1+\max(k_f, k_g)})$) is just too rapid for the interpolating property to imply \eqref{eq:extraLsym}.

\begin{remark}
 In the analogous case when $f$ is supersingular but $g$ is an ordinary CM form, these extra $L$-functions correspond to the extra two $p$-adic $L$-functions constructed in \cite{Loeffler-Heidelberg} using modular symbols for Bianchi groups.
\end{remark}

\subsection{Non-triviality of Beilinson--Flach elements}

\begin{corollary}
 \label{cor:nontrivialityofBFelements}
 If $|k_f-k_g| \geq 3$ then for each choice of $\lambda$ and $\mu$, the class $\res_p\left(\BF_{\lambda,\mu,1}\right)$ is non-trivial. 
\end{corollary}

\begin{proof}
 By symmetry, we may suppose that $k_f - k_g \ge 3$. Notice that the Euler product for the Rankin--Selberg $L$-series $L(f,g,s)$ converges absolutely at $s=k_f+1$ (since $k_f+1>\frac{k_f+k_g}{2}+2$). In particular, $L(f,g,k_f+1)$ is non-zero. For either value $\lambda \in \{\alpha_f, \beta_f\}$, the factors $(c^2 - \dots)$ and $ \mathcal{E}(f,g,1+k_f)$ appearing in Theorem~\ref{thm:LZgeominterpolationformula} are easily seen to be non-zero as well, using the fact that $\alpha_f, \beta_f$ are Weil numbers of weight $\frac{k_f+1}{2}>\frac{k_g+1}{2}+1$, whereas $\alpha_g, \beta_g$ are Weil numbers of weight $\frac{k_g+1}{2}$. Hence $L_p(f_\lambda, g)(k_f)$ is non-zero for both values of $\lambda$. By the explicit reciprocity law, this forces all four elements $\res_p\left(\BF_{\lambda,\mu,1}\right)$ to be non-zero.
\end{proof}

\begin{defn}
 For a character $\eta$ of $\Delta=\Gamma_{\textup{tor}}$, we let $e_{\eta}\in \LL$ denote the corresponding idempotent. 
\end{defn}

\begin{remark}
 \label{rem:nontrivialityofBFelements}
 One may argue as in the proof of Corollary~\ref{cor:nontrivialityofBFelements} to show that the projection $\res_p\left(e_{\omega^{j}}\BF_{\lambda,\mu,1}\right)$ is non-trivial for $1+\frac{k_f+k_g}{2}<j\leq \max(k_f,k_g)$.
\end{remark}

\begin{remark}
 Note that the interpolation formula for $ L_p^{\mathrm{geom}} (\cF, \cG)$ we have recorded in Theorem~\ref{thm:LZgeominterpolationformula} does not say anything about its value at $(r,r',j+\chi)$ where $\chi$ is a non-trivial finite order character of $p$-power conductor. This is the reason why we assume $|k_f-k_g|\geq 3$ in Corollary~\ref{cor:nontrivialityofBFelements}; with a stronger interpolation formula we could reduce this to $|k_f - k_g| \ge 2$, and even $|k_f - k_g| \ge 1$ conditionally on standard non-vanishing conjectures for complex $L$-functions.
 
 In the {sequel} \cite{BLV}, we {need} a similar non-vanishing result in the case $f = g$. In this particular situation, note that we have $k_f=k_g$ and the Rankin--Selberg $L$-series does not possess a single critical value. In order to prove the non-vanishing of the geometric $p$-adic $L$-function associated to the symmetric square, one needs to factor the $p$-adic Rankin--Selberg $L$-function as a product of the symmetric square $p$-adic $L$-function and a Kubota--Leopoldt $p$-adic $L$-function (extending the work of Dasgupta in the $p$-ordinary case). This is the subject of a forthcoming work of {Alessandro Arlandini}.
\end{remark}

\subsection{A partial result towards Conjecture~\ref*{conj:rank2}}

\begin{theorem}
 \label{thm:liftingBFattrivialtamelevelinFracH}
 Suppose that all four $p$-adic Rankin--Selberg $L$-functions \eqref{eq:fourLp} are non-zero-divisors in $\cH$, and that the following ``big image'' hypotheses hold:
 \begin{itemize}
  \item The representation $V$ is absolutely irreducible.
  \item There exists an element $\tau\in \Gal(\overline{\QQ}/\QQ(\mu_{p^\infty}))$ such that the $E$-vector space $V/(\tau-1)V$ is 1-dimensional.
  %\item There exists an element $\sigma\in \Gal(\overline{\QQ}/\QQ(\mu_{p^\infty}))$ which acts on $T$ by multiplication by $-1$.
 \end{itemize}
 
 Then there exists a class $\BF_1 \in \operatorname{Frac} \cH \otimes_{\Lambda} \bigwedge^2 H^1_{\Iw}(\QQ(\mu_{p^\infty}), T)$ satisfying
 \[ 
 \left\langle \cL_V(\BF_1), v_{\lambda, \mu}^* \right\rangle = \BF_{\lambda, \mu, 1}
 \]
 for all choices of $p$-stabilisations $\lambda$, $\mu$. In particular the ``extra'' anti-symmetry property \eqref{eq:extraLsym} holds.
\end{theorem}

The proof of this theorem will proceed in several steps.
\begin{proposition}\label{prop:Hproj}
 If any one of the four $p$-adic $L$-functions is a non-zero-divisor and the ``big image'' conditions hold, then $H^1_{\an}(\QQ(\mu_{p^\infty}), V)$ has rank 2 over $\cH$, and the map
 \[ H^1_{\an}(\QQ(\mu_{p^\infty}), V) \lra \cH^{\oplus 4} \]
 given by pairing $\cL_V \circ\, \loc_p$ with the four basis vectors $\{ v_{\alpha \alpha}^*, v_{\alpha \beta}^*, v_{\beta \beta}^*, v_{\beta \alpha}^*\}$ of $\Dcris(V^*)$ (in that order) is an injection.
\end{proposition}

\begin{proof} 
 Since the Perrin-Riou regulator $\cL_V$ is injective, it suffices to show that $H^1_{\an}(\QQ(\mu_{p^\infty}), V)$ has rank 2 and injects into $H^1_{\an}(\Qp(\mu_{p^\infty}), V)$ via $\loc_p$. Note that $H^1_{\an}(\QQ(\mu_{p^\infty}), V)$ is free thanks to our big image conditions, and its rank is at least $2$ by Tate's Euler characteristic formula.
 
 By symmetry we may suppose that $L(f_\alpha, g)$ is a non-zero-divisor. Choose a character $\chi$ of $\Gamma$ in each $\Delta$-isotypic component at which this $p$-adic L-function does not vanish, and away from the support of the torsion module $H^2_{\Iw}(\Qp, V)$. By the explicit reciprocity law, this implies that $\BF_{\alpha, \alpha, 1}(\chi)$ and $\BF_{\alpha, \beta, 1}(\chi)$ are both non-zero, and moreover that their images in $H^1(\Qp, V(\chi^{-1}))$ are linearly independent. By the Euler system machinery and an application of Poitou--Tate duality, precisely as in Theorem 8.2.1 and Corollary 8.3.2 of \cite{LZ1}, one sees that the relaxed Selmer group $H^1_{\mathrm{relaxed}}(\QQ, V(\chi^{-1}))$ is 2-dimensional and injects into the local cohomology at $p$. Since there is an injection 
 \[ H^1_{\an}(\QQ(\mu_{p^{\infty}}), V)/(\gamma - \chi(\gamma)) \hookrightarrow H^1_{\mathrm{relaxed}}(\QQ, V(\chi^{-1})),\]
 the result follows.
\end{proof}

We now assume, for the remainder of this section, that the hypotheses in the statement of Theorem \ref{thm:liftingBFattrivialtamelevelinFracH} are satisfied. Let $\mathcal{M}$ denote the image of $H^1_{\an}(\QQ(\mu_{p^\infty}), V)$ in $\cH^{\oplus 4}$ as given by Proposition~\ref{prop:Hproj}. Then $\mathcal{M}$ has rank 2, as we have just established. Given an element $x\in \cH^{\oplus 4}$ and $i\in\{1,2,3,4\}$, we write $x_i$ for its $i$-th coordinate (understood modulo 4, so that $x_5 = x_1$). For each $i$, let $\mathcal{M}^{(i)} = \{ x \in \mathcal{M}: x_i = 0\}$ be the kernel of the $i$-th coordinate projection.

We write $m^{(1)}, \dots, m^{(4)}$ for the images in $\mathcal{M}$ of the four analytic Iwasawa cohomology classes $\BF_1^{\alpha \alpha}, \BF_1^{\alpha \beta}, \BF_1^{\beta \beta}, \BF_1^{\beta \alpha}$ (in that order).

\begin{proposition}
 For each $i \in \{1, \dots, 4\}$, we have the following:
 \begin{itemize}
  \item $m^{(i)} \in \mathcal{M}^{(i)}$;
  \item $(m^{(i)})_{i + 1} = -(m^{(i + 1)})_i$, and this value is a non-zero-divisor;
  \item $\mathcal{M}^{(i)}$ has rank 1;
  \item $\mathcal{M}^{(i)} / \langle m^{(i)} \rangle$ is $\cH$-torsion.
  
 \end{itemize}
\end{proposition}

\begin{proof}
 The first two statements follow directly from the explicit reciprocity law (Theorem \ref{thm:explicitrecip}); the theorem in particular shows that the common value $(m^{(i)})_{i + 1} = -(m^{(i + 1)})_i$ is one of the four $p$-adic $L$-functions \eqref{eq:fourLp}, which are non-zero-divisors by assumption.
 
 In particular, this shows that the images of $\mathcal{M}$ under all four coordinate projections have rank $\ge 1$. Since $\mathcal{M}$ has rank 2, it follows that the submodules $\mathcal{M}^{(i)}$ all have rank 1, and that each $m^{(i)}$ spans a rank 1 submodule of $\mathcal{M}^{(i)}$, so the quotient $\mathcal{M}^{(i)}/ \langle m^{(i)} \rangle$ is torsion.
\end{proof}

\begin{proof}[Proof of Theorem \ref{thm:liftingBFattrivialtamelevelinFracH}]
 Let $\{ u, v \}$ denote the image in $\mathcal{M}$ of a basis of $H^1_{\Iw}(\QQ(\mu_{p^\infty}), T)$. Then $\{u, v\}$ is a basis of $\mathcal{M}$, and for each $i$, the vector
 \[ u_i \cdot v - v_i \cdot u \]
 lies in $\mathcal{M}^{(i)}$. It is also non-zero, since $u, v$ are linearly independent over $\cH$. Since $\mathcal{M}^{(i)}$ has rank 1, it follows that there is a non-zero-divisor $c_i$ in the total ring of fractions of $\cH$ such that
 \[ m^{(i)} = c_i \cdot (u_i \cdot v - v_i \cdot u).\]
 Substituting the definition of the $c_i$ into the formula $(m^{(i)})_{i + 1} = -(m^{(i + 1)})_i$, we deduce that
 \[ 
 c_i \cdot (u_i v_{i+1} - v_i  u_{i+1}) = c_{i+1} \cdot (u_i v_{i +1} - v_i  u_{i+1}), 
 \]
 and moreover that the common value is a non-zero-divisor. Hence $c_i = c_{i+1}$. Repeating this argument for each $i$, we see that the quantities $c_i$ are all equal to some common value $c \in \operatorname{Frac}(\cH)$. Let $\tilde{u}$ and $\tilde{v}$ be the preimages of $u$ and $v$ in $H^1_{\Iw}(\QQ(\mu_{p^\infty}), T)$. On unravelling the notation, we see that $m^{(i)}$ is equal to the image of $c \cdot \tilde u \wedge \tilde v \in \operatorname{Frac}(\cH) \otimes_{\Lambda} \bigwedge^2 H^1_{\Iw}(\QQ(\mu_{p^\infty}), T)$ under the map sending $x \wedge y$ to the $i$-th coordinate of $\left\langle \cL_V(x) y - \cL(y) x, v^{(i)}\right\rangle$, where $v^{(i)}$ is the $i$-th element of our basis $\left( v^*_{\alpha \alpha},  v^*_{\alpha \beta}, v^*_{\beta \beta}, v^*_{\beta \alpha}\right)$ of $\Dcris(V)^*$. So we may take $\mathrm{BF}_1 = c \cdot\tilde u \wedge\tilde v$.
\end{proof}

\begin{remark}
 We can carry out the same argument after applying the idempotent $e_\eta$, if we assume that the $e_\eta$ isotypic parts of the $L$-functions \eqref{eq:fourLp} are non-zero. It suffices to have non-vanishing of any three of the four, as is clear from the proof.
 
 More subtly, if one assumes that the symmetry property \eqref{eq:extraLsym} for the ``extra'' $L$-functions is true, then one can prove the same theorem assuming that two of the four functions \eqref{eq:fourLp} and one of the ``extra'' $L$-functions is non-zero. This non-vanishing can be deduced from the explicit reciprocity law when $|k_f - k_g| \ge 3$ and appropriate $\eta$, as above.
\end{remark}

\section{Logarithmic matrix and factorisations}

In the following sections of the paper, we explore some of the consequences of Conjecture \ref{conj:rank2}, and show that it implies factorisations of the Beilinson--Flach elements via a matrix of logarithms.

\subsection{Integral $p$-adic Hodge theory for $V$}
In this section, we assume (as in the introduction) that both $f$ and $g$ are  non-ordinary at $p$. We also impose the following \emph{Fontaine--Laffaille} hypothesis:
\[
p>k_f+k_g+2.
\]
We describe our constructions assuming $k_f \le k_g$ for simplicity; the case $k_f \ge k_g$ is similar.

For $h\in\{f,g\}$, we have the basis $\omega_h$, $\vp(\omega_h)$ of $\Dcris(R_h^*)$ with $\omega_h$ generating $\Fil^0\Dcris(R_h^*)$. As worked out in \cite[\S3.1]{LLZCJM}, the matrix of $\vp$ with respect to this basis is
\[
A_h=\begin{pmatrix}
0 &-\frac{\epsilon_h(p)}{p^{k_h+1}}\\
1&\frac{a_p(h)}{p^{k_h+1}}
\end{pmatrix}.
\]

Let $T$ be the representation $R_f^*\otimes R_g^*$.  Then we consider the basis $v_1=\omega_f\otimes \omega_g$, $v_2=\omega_f\otimes \vp(\omega_g)$, $v_3=\vp(\omega_f)\otimes\omega_g$, $v_4=\vp(\omega_f)\otimes\vp(\omega_g)$ for $\Dcris(T)$. We note in particular that it respects the filtration of $\Dcris(T)$ in the following sense:
\begin{equation}\label{eq:filtration}
 \Fil^i\Dcris(T)=
 \begin{cases}
  \langle v_1,v_2,v_3,v_4\rangle& i\le -k_f-k_g-2,\\
  \langle v_1,v_2,v_3\rangle &-k_f-k_g-1 \le i\le-k_g-1,\\
  \langle v_1,v_2\rangle& -k_g\le i\le -k_f-1,\\
  \langle v_1\rangle &-k_f\le i\le 0,\\
  0& i\ge 1.
 \end{cases} 
\end{equation}

The matrix of $\vp$ with respect to this basis is
\[
A=A_0\cdot\begin{pmatrix}
1\\
&\frac1{p^{k_f+1}}\\
&&\frac1{p^{k_g+1}}\\
&&&\frac1{p^{k_f+k_g+2}}
\end{pmatrix},
\]
where $A_0$ is the matrix defined by
\[
\begin{pmatrix}
0&0&0&\epsilon_f(p)\epsilon_g(p)\\
0&0&-\epsilon_f(p)&-\epsilon_f(p)a_p(g)\\
0&-\epsilon_g(p)&0&-\epsilon_g(p)a_p(f)\\
1&a_p(g)&a_p(f)&a_p(f)a_p(g)\\
\end{pmatrix}.
\]
{
 We introduce the following convention.
 \begin{conv}\label{conv}
  Let $n\ge 1$ be an integer and $U$ an $E$-vector space. If $M=(m_{ij})$ is an $n\times n$ matrix defined over $E$ and $u_1,\ldots u_n$ are elements in $U$, we write
  \[
  \begin{pmatrix}
  u_1&\cdots &u_n
  \end{pmatrix}\cdot M
  \]
  for the row vector of elements in $U$ given by $\sum_{i=1}^nu_im_{ij}$, $j=1,\ldots,n$.
 \end{conv}
 Under this convention, we have the equation
 \begin{equation}\label{eq:multipyA}
  \begin{pmatrix}
   \vp(v_1)&\vp(v_2)&\vp(v_3)&\vp(v_4)
  \end{pmatrix}=
  \begin{pmatrix}
   v_1&v_2&v_3&v_4
  \end{pmatrix}\cdot A.
 \end{equation}
}

Recall that the Wach module $\NN(T)$ is a free module of rank 4 over $\AA^+_{\Qp}$, equipped with a canonical isomorphism 
\begin{equation}\label{eq:comparison1}
 \NN(T) / \pi \NN(T) \cong \Dcris(T).
\end{equation}
By \cite[proof of Proposition V.2.3]{berger04} (see also \cite[Proposition~4.1]{leitohoku}), our Fontaine--Laffaille hypothesis allows us to lift the basis $\{v_i\}$ of $\Dcris(T)$ as an $\cO$-module to a basis $\{n_i\}$ of $\NN(T)$ as an $\AA^+_{\Qp}$-module, and the matrix of $\vp$ with respect to the basis $\{n_i\}$ is given by
\[
P\coloneqq A_0\cdot \begin{pmatrix}
\mu^{k_f+k_g+2}\\
&\frac{\mu^{k_g+1}}{q^{k_f+1}}\\
&&\frac{\mu^{k_f+1}}{q^{k_g+1}}\\
&&&\frac1{q^{k_f+k_g+2}}
\end{pmatrix},
\]
where $\mu=\frac{p}{q-\pi^{p-1}}\in 1 + \pi \AQp$. Note that $P^{-1}$ is integral. {Furthermore, similar to the equation \eqref{eq:multipyA}, we have
 \begin{equation}\label{eq:mutliplyP}
  \begin{pmatrix}
   \vp(n_1)&\vp(n_2)&\vp(n_3)&\vp(n_4)
  \end{pmatrix}=
  \begin{pmatrix}
   n_1&n_2&n_3&n_4
  \end{pmatrix}\cdot P.
 \end{equation}
}
\subsection{The logarithmic matrix}

There is an isomorphism
\begin{equation} 
 \Brig\left[ \tfrac{t}{\pi}\right] \otimes_{\AA^+_{\Qp}} \NN(T) \cong \Brig\left[ \tfrac{t}{\pi}\right] \otimes_{\Zp} \Dcris(T)
\end{equation}
compatible with \eqref{eq:comparison1} via reduction mod $\pi$. Let $M \in \operatorname{GL}_4\left(\Brig\left[ \tfrac{t}{\pi}\right]\right)$ be the matrix of this isomorphism with respect to our bases $\{v_i\}$ and $\{n_i\}$, so that
\begin{equation}\label{eq:matricepassage}
 \begin{pmatrix}
  n_1&n_2&n_3&n_4
 \end{pmatrix}=\begin{pmatrix}
  v_1&v_2&v_3&v_4
 \end{pmatrix}\cdot M
\end{equation}
{under Convention~\ref{conv}.}
By \cite[Proposition~4.2]{leitohoku}, we can (and do) choose the $n_i$ such that 
\begin{equation}\label{eq:congruentid}
 M\equiv I_4\mod \pi^{k_f+k_g+2}.
\end{equation}
If we apply $\vp$, {we deduce from \eqref{eq:multipyA} and \eqref{eq:mutliplyP} that}
\[
M=A\vp(M)P^{-1}.
\]
If we repeatedly apply $\vp$, we get
\[
M=A^n\vp^n(M)\vp^{n-1}(P^{-1})\cdots \vp(P^{-1})P^{-1}.
\]
So, in particular,
\begin{equation}\label{eq:congruentpi}
 M\equiv A^n\vp^{n-1}(P^{-1})\cdots \vp(P^{-1})P^{-1}\mod \vp^n(\pi^{k_f+k_g+2})
\end{equation}
thanks to \eqref{eq:congruentid}.
We define the \emph{logarithmic matrix} to be the $4 \times 4$ matrix over $\cH$ given by
\[
\Mlog\coloneqq\fM^{-1}\Big((1+\pi)A\vp(M)\Big),
\]
where $\fM$ is the Mellin transform (applied individually to each entry of the matrix $(1 + \pi) A\varphi(M)$). Recall from \cite[\S3]{leiloefflerzerbes11} that, up to a unit, the determinant of $\Mlog$ is given by
\[
\fn_{k_f+1}\cdot\fn_{k_g+1}\cdot \fn_{k_f+k_g+2}.
\]
Furthermore, by Theorem~\ref{thm:mellin} and \eqref{eq:congruentpi}, we have the congruence
\begin{equation}\label{eq:congruentMlog}
 \Mlog\equiv A^{n+1}\cdot H_n\mod \omega_{n,k_f+k_g+2},
\end{equation}
where $H_n=\fM^{-1}(\vp^{n}(P^{-1})\cdots \vp(P^{-1}))$.

\begin{lemma}\label{lem:adj}
 The adjugate  matrix $\adj(\Mlog) = \det(\Mlog) \Mlog^{-1}$ is divisible by $\fn_{k_f+1}\fn_{k_g+1}$.
\end{lemma}
\begin{proof}
 By \eqref{eq:congruentMlog}, we have 
 \[
 \Mlog\equiv A^{n+1}\cdot H_n\mod\Phi_{n,k_f+k_g+2}.
 \]
 So, it is enough to show that $\adj(H_n)$ is divisible by $\Phi_{n,k_f+1}\Phi_{n,k_g+1}$ for all $n$. {Recall that,
  \[
  \fM(H_n)=\vp^{n}(P^{-1})\cdots \vp(P^{-1}).
  \]
  From the construction of $P$,  the last three rows of $P^{-1}$ are divisible by 
  $q^{k_f+1}$, $q^{k_g+1}$ and $q^{k_f+k_g+2}$ respectively. Therefore, the last three rows of $\fM(H_n)$ are  divisible by $\vp^{n}(q^{k_f+1})$, $\vp^{n}(q^{k_g+1})$ and $\vp^{n}(q^{k_f+k_g+2})$ respectively. Theorem~\ref{thm:mellin} then tells us that the last three rows} of $H_n$ are divisible by $\Phi_{n,k_f+1}$, $\Phi_{n,k_g+1}$ and $\Phi_{n,k_f+k_g+2}$. Hence, when we take adjugate, every entry will be divisible by $\Phi_{n,k_f+1}\Phi_{n,k_g+1}$ as required.
\end{proof}

Let $\{v_{\lambda,\mu}\}_{\lambda,\mu\in\{\alpha,\beta\}}$ be the eigenvector basis of $\Dcris(V)$ as given in \S\ref{S:conjrank2ES}. The  matrix of $\varphi$ with respect to this basis is
\[ D \coloneqq \begin{pmatrix}
\frac1{\alpha_f\alpha_g}\\
&\frac1{\alpha_f\beta_g}\\
&&\frac1{\beta_f\alpha_g}\\
&&&\frac1{\beta_f\beta_g}
\end{pmatrix}.\]
Recall that we defined $v_{\lambda\mu}=v_{f,\lambda}^*\otimes v_{g,\mu}^*$, where $v_{h,\alpha}$ and $v_{h,\beta}$ are eigenvectors in $\Dcris(V_h)$ with $v_{h,\alpha}=v_{h,\beta}\mod \Fil^1$ for $h\in\{f,g\}$. Since $\langle v_{h,\alpha}^*+v_{h,\beta}^*,v_{h,\alpha}-v_{h,\beta}\rangle =0$ by duality, we have $v_{h,\alpha}^*+v_{h,\beta}^*=0\mod \Fil^0$. After multiplying $\omega_h$ by a scalar if necessary, we may choose $$v_{h,\alpha}^*=\alpha_h(\omega_h-\beta_h\vp(\omega_h)),\quad\text{ and }\quad v_{h,\beta}^*=-\beta_h(\omega_h-\alpha_h\vp(\omega_h)).$$  We let $Q$ be the change-of-basis matrix from the basis $\{v_i\}$ to this eigenvector basis, so that $D = Q^{-1}AQ$. Explicitly, we have
\[
Q=
\begin{pmatrix}
\alpha_f\alpha_g&-\alpha_f\beta_g&-\beta_f\alpha_g&\beta_f\beta_g\\
-\alpha_f\alpha_g\beta_g&\alpha_f\alpha_g\beta_g&-\alpha_g\beta_f\beta_g&\alpha_g\beta_f\beta_g\\
-\alpha_f\alpha_g\beta_f&\alpha_f\beta_f\beta_g&-\alpha_f\alpha_g\beta_f&\alpha_f\beta_f\beta_g\\
\alpha_f\alpha_g\beta_f\beta_g&-\alpha_f\alpha_g\beta_f\beta_g&-\alpha_f\alpha_g\beta_f\beta_g&\alpha_f\alpha_g\beta_f\beta_g
\end{pmatrix}.
\]
Using this matrix, we may rewrite \eqref{eq:congruentMlog} as
\begin{equation}\label{eq:QMlog}
 Q^{-1}\Mlog\equiv D^{n+1}Q^{-1}H_n\mod\omega_{n,k_f+k_g+2}.
\end{equation}

\begin{lemma}\label{lem:growth}
 The entries in the first column of $Q^{-1}\Mlog$ are all $O(\log_p^{v_p(\alpha_f\alpha_g)})$, and similarly for the other three columns.
\end{lemma}
\begin{proof}Recall that $H_n$ is a matrix defined over $\Lambda$.
 The congruence relation \eqref{eq:QMlog} tells us that the entries of the first column of
 $Q^{-1}\Mlog$ modulo $\omega_{n,k_f+k_g+2}$  have denominator $O(p^{v_p(\alpha_f\alpha_g)})$. Hence, our result follows from \cite[Lemma~2.2]{BL16b}, which is a slight generalization of \cite[\S1.2.1]{perrinriou94}. 
\end{proof}

\subsection{Characterising the image}

Here we prove an important linear-algebra result describing the $\Lambda$-submodule of $\cH^{\oplus 4}$ generated by the logarithmic matrix; we shall see that it consists exactly of those elements which ``look like'' they are in the image of the Perrin-Riou regulator map.

\begin{proposition}\label{prop:vanish}
 Let $F_{\lambda,\mu}\in\cH$, $\lambda,\mu\in\{\alpha,\beta\}$, be four functions. Suppose that, for some integer $j \in \{ 0, \dots, k_f+k_g+1\}$ and some Dirichlet character $\theta$ of conductor $p^n$ with $n>1$, we have
 \[
 \sum_{\lambda,\mu}(\lambda\mu)^nF_{\lambda,\mu}(\chi^j\theta)v_{\lambda,\mu}\in \Qpn\otimes \Fil^{-j}\Dcris(T).
 \]
 Then, 
 \[
 \frac{\adj(Q^{-1}\Mlog)}{\fn_{k_f+1}\fn_{k_g+1}}\cdot \begin{pmatrix}
 F_{\alpha,\alpha}\\F_{\alpha,\beta}\\F_{\beta,\alpha}\\F_{\beta,\beta}
 \end{pmatrix}(\chi^j\theta)=0.
 \]
\end{proposition}
\begin{proof}
 When we evaluate an element in $\Lambda$ at $\chi^j\theta$, the result only depends on the given element modulo $\Tw^{-j}\Phi_{n-1}(X)$.
 By \eqref{eq:congruentMlog},
 \[
 \adj(Q^{-1}\Mlog)\equiv \adj(D^nQ^{-1}H_{n-1})\equiv\frac{\det(D^n)}{\det(Q)}\adj(H_{n-1})QD^{-n}\mod \Tw^{-j}\Phi_{n-1}.
 \]

 We recall from the proof of Lemma~\ref{lem:adj} that the last three {rows} of $H_{n-1}$ are divisible by $\Phi_{n-1,k_f+1}$, $\Phi_{n-1,k_g+1}$ and $\Phi_{n-1,k_f+k_g+2}$ respectively.   So, after dividing $\fn_{k_f+1}\fn_{k_g+1}$, the first column of $\adj(H_{n-1})$ is divisible by $\Phi_{n-1,k_f+k_g+2}$, the second column is divisible by $\Phi_{n-1,k_f+k_g+2}/\Phi_{n-1,k_f+1}$, whereas the third one is divisible by $\Phi_{n-1,k_f+k_g+2}/\Phi_{n-1,k_g+1}$. In particular, when evaluated at a character of the form $\chi^j\theta$, we have
 \[
 \frac{\adj(H_{n-1})}{\fn_{k_f+1}\fn_{k_g+1}}(\chi^j\theta)=
 \begin{pmatrix}
 0&0&0&*\\
 0&0&0&*\\
 0&0&0&*\\
 0&0&0&*
 \end{pmatrix}
 \]
 if $k_g+1\le j\le k_f+k_g+1$. When $j$ is in this range,  our assumption on $F_{\lambda,\mu}$ tells us that
 \[
 QD^{-n}\begin{pmatrix}
 F_{\alpha,\alpha}\\
 F_{\alpha,\beta}\\
 F_{\beta,\alpha}\\
 F_{\beta,\beta}
 \end{pmatrix}(\chi^j\theta)=\begin{pmatrix}
 *\\
 *\\
 *\\
 0
 \end{pmatrix}
 \]
 thanks to the description of the filtration in \eqref{eq:filtration}.
 The result then follows from multiplying the two equations above.
 
 For the other cases, we have 
 \[
 \frac{\adj(H_{n-1})}{\fn_{k_f+1}\fn_{k_g+1}}(\chi^j\theta)=\begin{pmatrix}
 0&0&*&*\\
 0&0&*&*\\
 0&0&*&*\\
 0&0&*&*
 \end{pmatrix},\quad
 QD^{-n}\begin{pmatrix}
 F_{\alpha,\alpha}\\
 F_{\alpha,\beta}\\
 F_{\beta,\alpha}\\
 F_{\beta,\beta}
 \end{pmatrix}(\chi^j\theta)=\begin{pmatrix}
 *\\
 *\\
 0\\
 0
 \end{pmatrix}
 \]
 if $k_f-1\le j\le k_g $ and
 \[
 \frac{\adj(H_{n-1})}{\fn_{k_f+1}\fn_{k_g+1}}(\chi^j\theta)=\begin{pmatrix}
 0&*&*&*\\
 0&*&*&*\\
 0&*&*&*\\
 0&*&*&*
 \end{pmatrix},\quad
 QD^{-n}\begin{pmatrix}
 F_{\alpha,\alpha}\\
 F_{\alpha,\beta}\\
 F_{\beta,\alpha}\\
 F_{\beta,\beta}
 \end{pmatrix}(\chi^j\theta)=\begin{pmatrix}
 *\\
 0\\
 0\\
 0
 \end{pmatrix}
 \]
 if $0\le j\le k_f$, so we are done.
\end{proof}

\begin{theorem}
 Let $F_{\lambda,\mu}\in\cH$, $\lambda,\mu\in\{\alpha,\beta\}$ be four functions such that for all integers $0\le j \le k_f+k_g+1$ and all Dirichlet characters $\theta$ of conductor $p^n$ with $n>1$,
 \[
 \sum_{\lambda,\mu}(\lambda\mu)^nF_{\lambda,\mu}(\chi^j\theta)v_{\lambda,\mu}\in \Qpn\otimes \Fil^{-j}\Dcris(T).
 \] Then, 
 \[
 \begin{pmatrix}
 F_{\alpha,\alpha}\\F_{\alpha,\beta}\\F_{\beta,\alpha}\\F_{\beta,\beta}
 \end{pmatrix}=Q^{-1}\Mlog\cdot \begin{pmatrix}
 F_{\#,\#}\\F_{\#,\flat}\\F_{\flat,\#}\\F_{\flat,\flat}
 \end{pmatrix}
 \]
 for some $F_{\bc}\in \cH$. 
 
 Furthermore, if $F_{\lambda,\mu}=O(\log_p^{v_p(\lambda_f\mu_g)})$ for all four choices of $\lambda$ and $\mu$, then $F_{\bc}=O(1)$ for all $\bullet$ and $\circ$.
\end{theorem}
\begin{proof}
 Proposition~\ref{prop:vanish} tells us that 
 \[
 \frac{\adj(Q^{-1}\Mlog)}{\fn_{k_f+1}\fn_{k_g+1}}\cdot \begin{pmatrix}
 F_{\alpha,\alpha}\\F_{\alpha,\beta}\\F_{\beta,\alpha}\\F_{\beta,\beta}
 \end{pmatrix}\in \fn_{k_f+k_g+2}\cH^{\oplus 4}.
 \]
 But since  the determinant of $Q^{-1}\Mlog$ is up to a unit $\fn_{k_f+1}\fn_{k_g+1}\fn_{k_f+k_g+2}$, 
 \[
 \frac{\adj(Q^{-1}\Mlog)}{\fn_{k_f+1}\fn_{k_g+1}\fn_{k_f+k_g+2}}
 \]
 is (again up to a unit) $(Q^{-1}\Mlog)^{-1}$, hence the decomposition as claimed.
 
 If furthermore $F_{\lambda,\mu}=O(\log_p^{v_p(\lambda_f\mu_g)})$, then Lemma~\ref{lem:growth} tells us that all entries in the product 
 \[
 \frac{\adj(Q^{-1}\Mlog)}{\fn_{k_f+1}\fn_{k_g+1}}\cdot \begin{pmatrix}
 F_{\alpha,\alpha}\\F_{\alpha,\beta}\\F_{\beta,\alpha}\\F_{\beta,\beta}
 \end{pmatrix}
 \]
 are $O(\log_p^{k_f+k_g+2})$. This says that the quotient 
 \[
 \frac{\adj(Q^{-1}\Mlog)}{\fn_{k_f+1}\fn_{k_g+1} \fn_{k_f+k_g+2}}\cdot \begin{pmatrix}
 F_{\alpha,\alpha}\\F_{\alpha,\beta}\\F_{\beta,\alpha}\\F_{\beta,\beta}
 \end{pmatrix}
 \]
 is in $O(1)$ and we are done.
\end{proof}

\begin{remark}
 Note that the condition on $F_{\alpha\beta}$ is automatically satisfied if the $F_{\alpha\beta}$ are the components in our eigenvector basis of an element of the Perrin-Riou regulator map, since the regulator interpolates the Bloch--Kato dual exponential for $j \ge 0$, and the dual exponential map for $T(-j)$ factors through $\Fil^{-j} \Dcris$. The above result should be viewed as a sort of converse to this statement, showing that these vanishing conditions force a factorisation via the matrix of logarithms, of the same form as the factorisation established for the Perrin-Riou regulator in \cite{leiloefflerzerbes10}.
\end{remark}

\section{Equivariant Perrin-Riou maps and $(\#,\flat)$-splitting}
\subsection{Perrin-Riou maps and signed Coleman map}

Let $f$ and $g$ be two modular forms as in the previous section. We shall write $T=R_f^*\otimes R_g^*$ as before. 
Let $F/\Qp$ be a finite unramified extension. We write $\NN_F(T)$ and $\Dcris(F,T)$ for the Wach module and Dieudonn\'e module of $T$ over $F$. We have already fixed bases $\{n_i\}$ and $\{v_i\}$ of $\NN_{\Qp}(T)$ and $\Dcris(\Qp,T)$  respectively. Given that 
$$\NN_F(T)=\cO_F \otimes_{\Zp} \NN_{\Qp}(T),\quad \Dcris(F,T)=\cO_F\otimes_{\Zp} \Dcris(\Qp,T),$$
we may extend the bases we have chosen to  $\NN_F(T)$ and $\Dcris(F,T)$ naturally.

We recall from \eqref{eq:matricepassage} that the change of basis matrix $M$ between the two bases above results in a logarithmic matrix $\Mlog$. Furthermore, given that $M\equiv I_4\mod \pi^2$ by \eqref{eq:congruentpi}, \cite[Theorem~2.5]{LLZCJM} says that $\{(1+\pi)\vp(n_i)\}$ is a $\Lambda$-basis of $\vp^*(\NN(T))$. We recall from \cite[Remark~3.4]{leiloefflerzerbes10} that $(1-\vp)\NN_F(T)^{\psi=1}\subset (\vp^*\NN_F(T))^{\psi=0}$. This allows us to define four Coleman maps $\col_{F,\bc}:\NN_F(T)^{\psi=1}\rightarrow \cO_F\otimes\Lambda$ via the relation
\[
(1-\vp)z=\begin{pmatrix}
v_1&v_2&v_3&v_4
\end{pmatrix}\cdot \Mlog\cdot \begin{pmatrix}
\col_{F,\#,\#}(z)\\
\col_{F,\#,\flat}(z)\\
\col_{F,\flat,\#}(z)\\
\col_{F,\flat,\flat}(z)
\end{pmatrix}
\]
for  $z\in \NN_F(T)^{\psi=1}$ (see  \cite[\S3]{leiloefflerzerbes11} for details).

A result of Berger \cite[Theorem~A.3]{berger03} tells us that there is an isomorphism 
\[
h^1_{F,T}: \NN_F(T)^{\psi=1}\rightarrow\HIw(F,T).
\]
We shall abuse notation and denote $\col_{F,\bc}\circ (h^1_{F,T})^{-1}$ by simply $\col_{F,\bc}$ for $\bc\in\{\#,\flat\}$.
The Perrin-Riou regulator map
\[
\cL_{T,F}:\HIw(F,T)\rightarrow \cH\otimes \Dcris(F,T)
\]
is given by
\[
(\fM^{-1}\otimes 1)\circ(1-\vp)\circ (h^1_{F,T})^{-1}.
\]
Hence, we have the decomposition
\begin{equation}\label{eq:decompPR}
 \cL_{T,F}(z)=\begin{pmatrix}
  v_1&v_2&v_3&v_4
 \end{pmatrix}\cdot \Mlog\cdot\begin{pmatrix}
  \col_{F,\#,\#}(z)\\
  \col_{F, \#,\flat}(z)\\
  \col_{F, \flat,\#}(z)\\
  \col_{F, \flat,\flat}(z)
 \end{pmatrix}.
\end{equation}

Exactly as in \eqref{eq:equivariantreg}, for $m \ge 1$ we can combine the maps $\col_{\QQ(\mu_m)_v, \bc}$ for primes $v \mid p$ of $\QQ(\mu_m)$, and the map $\nu_m$, to obtain maps
\[ \col_{m, \bc}: H^1_{\Iw}(\QQ(\mu_{mp^\infty}) \otimes \Qp, T) \to \Lambda_m. \]

\subsection{Signed Selmer groups}

Let $\Lambda^{\iota}$ be the free rank 1 $\Lambda$-module on which $G_{\QQ}$ acts via the inverse of the canonical character $G_{\QQ} \twoheadrightarrow \Gamma \hookrightarrow \Lambda^\times$. We write $\TT \coloneqq T\otimes \LL^\iota$, and we define the (compact) \emph{signed Selmer group} $H^1_{\cFbc}(\QQ(\mu_m),\TT)$ by setting
$$H^1_{\cFbc}(\QQ(\mu_m),\TT)\coloneqq\ker\left(H^1(\QQ(\mu_m),\TT)\longrightarrow \prod_{v|p}\frac{H^1(\QQ(\mu_m)_v,\TT)}{\ker\left(\col_{\bc, \QQ(\mu_m)_v}\right)}\right)\,.$$

We next define discrete signed Selmer groups for the dual Galois representation $T^\vee(1)$. 
Let $F/\Qp$ be a finite unramified extension. By Tate duality, there is a perfect pairing
\[
\HIw(F,T)\times H^1(F(\mu_{p^\infty}),T^\vee(1))\rightarrow \Qp/\Zp.
\]
For $\bc\in\{\#,\flat\}$, we define
\[
H^1_{\bc}(F(\mu_{p^\infty}),T^\vee(1))\subset H^1(F(\mu_{p^\infty}),T^\vee(1))
\]
to be the orthogonal complement of $\ker\left(\col_{\bc, F}\right)$. 

\begin{defn}
 The discrete signed Selmer group $\Sel_{\bc}(T^\vee(1)/\QQ(\mu_{mp^\infty}))$ is the kernel of the restriction map
 \[
 H^1(\QQ(\mu_{mp^\infty}),T^\vee(1))\lra \prod_{v|p}\frac{H^1(\QQ(\mu_{mp^\infty})_v,T^\vee(1))}{H^1_{\bc}(\QQ(\mu_{mp^\infty})_v,T^\vee(1))} \times \prod_{v\nmid p}\frac{H^1(\QQ(\mu_{mp^\infty})_v,T^\vee(1))}{H^1_f(\QQ(\mu_{mp^\infty})_v,T^\vee(1))},
 \]
 where $v$ runs through all primes of $\QQ(\mu_{mp^\infty})$.
\end{defn}

\subsection{$(\#,\flat)$-splitting and rank-2 Euler systems}

Our goal in this section is to formulate a weaker alternative to Conjecture~\ref{conj:rank2}, which we are able to verify in many cases of interest (in \cite{BL16b, BLV}), allowing us to make full use of the Euler system machinery in these scenarios. 

\begin{conjecture}
 \label{conj:signedfactBF}
 There exists a non-zero $r_0 \in \ZZ$, and a collection of elements $\BF_{\bc,m}\in H^1_{\cFbc}(\QQ(\mu_m),\TT)$ for each $m \in \NP$ and each choice of $\bc\in\{\#,\flat\}$, such that
 \begin{equation}
  \label{eqn:ESrank2impliessignedfactorizationforBF}
  r_0\cdot\begin{pmatrix}
   \BF_{\alpha,\alpha,m}\\ \BF_{\alpha,\beta,m}\\ \BF_{\beta,\alpha,m}\\ \BF_{\beta,\beta,m}
  \end{pmatrix}=Q^{-1}\Mlog\cdot \begin{pmatrix}
   \BF_{\#,\#,m}\\ \BF_{\#,\flat,m}\\ \BF_{\flat,\#,m}\\ \BF_{\flat,\flat,m}
  \end{pmatrix}\,.
 \end{equation}
\end{conjecture}

Note that the $\BF_{\bc, m}$ are uniquely determined if they exist, since the determinant of $Q^{-1} \Mlog$ is a non-zero-divisor in $\cH$.

\begin{proposition}
 \label{prop:Descendfromrank2ESproducesrank2ES}
 If Conjecture \ref{conj:rank2} holds, then Conjecture \ref{conj:signedfactBF} holds (and we may take $r_0 = 1$).
\end{proposition}

\begin{proof}
 Suppose that Conjecture~\ref{conj:rank2} is true. For each $m\in \NP$ and each $\bc \in \{ \#, \flat\}$, we can regard $\col_{m, \bc} \circ \loc_p$ as a map
 \[
 \bigwedge^2 H^1_\Iw(\QQ(\mu_{mp^\infty}), T)\longrightarrow  H^1_\Iw(\QQ(\mu_{mp^\infty}), T).
 \]
 Let us set $\BF_{\bc,m} \coloneqq \col_{\bc,m}(\BF_m) \in H^1_\Iw(\QQ(\mu_m),T)$, where $\BF_m$ is the element of Conjecture \ref{conj:rank2}. Then it is clear that $\BF_{\bc,m}\in H^1_{\cFbc}(\QQ(\mu_m),\TT)$; and the formula \eqref{eq:decompPR} relating the Perrin-Riou regulator $\cL_{V, m}$ to the Coleman maps implies that we have
 \[ 
 \begin{pmatrix}
 \langle \cL_{m, V}(\BF_m), v_{\alpha, \alpha}^*\rangle \\[1mm] 
 \langle \cL_{m, V}(\BF_m), v_{\alpha, \beta}^*\rangle\\[1mm] 
 \langle \cL_{m, V}(\BF_m), v_{\beta, \alpha}^*\rangle\\[1mm]
 \langle \cL_{m, V}(\BF_m), v_{\beta, \beta}^*\rangle
 \end{pmatrix}=Q^{-1}\Mlog\cdot 
 \begin{pmatrix}
 \BF_{\#,\#,m} \\ \BF_{\#,\flat,m} \\ \BF_{\flat,\#,m} \\ \BF_{\flat,\flat,m}
 \end{pmatrix},
 \]
 where $v^*_{\alpha\alpha}, \dots, v^*_{\beta\beta}$ is the eigenvector basis of $\Dcris(V^*)$. By the defining property of $\BF_m$, we have $\langle \cL_{V, m}\left(\BF_m\right), v_{\lambda,\mu}^* \rangle = \BF_{\lambda, \mu, m}$ for each $\lambda, \mu$, which gives the required factorisation.
\end{proof}

%
%\begin{remark}
%\label{rem:Descendfromrank2ESproducesrank2ES}
% We assume the validity of Conjecture~\ref{conj:rank2} also in this remark. Let $\BF_{\bc,m}$ be as in Remark~\ref{rem:ESrank2existsimpliesfactorization}.  By Proposition~\ref{prop:PRrankreduction} $($applied with $\Phi_m\coloneqq \col_{\bc,m}$ and $c_p(m)=\BF_m)$, it follows that $\BF_{\bc,m}$ is an Euler system of rank one. We note that it was Otsuki who has first made use of Coleman maps for rank reducing Euler systems.
%Notice further that we have 
%\[
%\col_{\bc,m}\left(\BF_{\bc,m}\right)=0.
%\]
%by construction. In other words, $\BF_{\bc,m} \in H^1_{\cFbc}(\QQ(\mu_m),\TT)$
%and  the collection $\{\BF_{\bc,m}\}_m$ is a \emph{locally restricted Euler system}.
%\end{remark}

Consider the following anti-symmetry condition:
\\\\
{\bf{(A--Sym)}}  For all possible choices of the symbols $\ts, \bc \in \{\#, \flat\}$ and every $m\in \mathcal{N}(\mathcal{P})$ we have
$$\col_{m,\ts}\circ\res_p\left(\BF_{\bc,m}\right)=-\col_{m,\bc}\circ\res_p\left(\BF_{\ts,m}\right).$$
\\
\begin{remark}
 \label{rmk:locrestrict}
 Assume that Conjecture~\ref{conj:signedfactBF} and the hypothesis {\bf{(A--Sym)}} hold true. Then for each choice of $\bc\in\{\#,\flat\}$, the collection $\left\{\BF_{\bc,m}\right\}_m$ is a (rank 1) \emph{locally restricted Euler system} in the sense of \cite[Appendix A]{kbbleiPLMS}, since each collection of $p$-stabilized Beilinson--Flach classes $\{\BF_{\lambda,\mu,m}\}_m$ $($for $\lambda,\mu\in \{\alpha,\beta\})$ verifies the Euler system distribution relations as $m$ varies.
 
 See Section~\ref{subsec:partialsignedBFsplitting}, where we partially verify Conjecture~\ref{conj:signedfactBF} and Proposition~\ref{prop:antisymmetryofmatrixofsignedpadicLfunctions}, where we give a sufficient condition for the validity of  {\bf{(A--Sym)}}. In the sequel \cite{BLV}, we prove an appropriate variant of this conjecture for the twist $\textup{Sym}^2f\otimes \chi$ of the symmetric square motive with a Dirichlet character. 
 
 Finally, we note that a factorization similar to \eqref{eqn:ESrank2impliessignedfactorizationforBF} is proved in \cite{BL16b} unconditionally when the newform $g$ is taken to be a $p$-ordinary CM form. 
\end{remark}

\begin{proposition}
 \label{prop:antisymmetryofmatrixofsignedpadicLfunctions}
 Suppose one of the following conditions:
 \begin{enumerate}
  \item[(i)] Conjecture \ref{conj:rank2} holds;
  \item[(ii)] Conjecture \ref{conj:signedfactBF} holds and the hypotheses of Theorem \ref{thm:liftingBFattrivialtamelevelinFracH} are satisfied for Rankin--Selberg convolutions $f\otimes g\otimes\eta$, where $\eta$ runs through  characters of $\Gal(\QQ(m)/\QQ)$ with $m\in \mathcal{N}(\mathcal{P})$.
 \end{enumerate}
 Then {\bf{(A--Sym)}} holds true.
\end{proposition}

\begin{proof}
 If (i) holds, then the elements $\BF_{\bc,m}$ must arise as the images of $\BF_1$ under the Coleman maps, as in the preceding proposition (by the uniqueness of the factorisation \eqref{eqn:ESrank2impliessignedfactorizationforBF}); the above symmetry property then obvious. If we assume (ii) holds, then we may carry out exactly the same argument on passing to $\eta$-isotypic components (where $\eta$ runs through characters of $\Gal(\QQ(m)/\QQ)$) and after extending scalars to $\operatorname{Frac} \cH$.
\end{proof}

\subsection{Partial $(\#,\flat)$-splitting of Beilinson--Flach classes}
\label{subsec:partialsignedBFsplitting}
In this section, we give evidence towards Conjecture~\ref{conj:signedfactBF} by proving a partial $(\#,\flat)$-factorization of Beilinson--Flach classes.

Let $m\ge1$ be an integer coprime to $p$. We write $$\BF_{\lambda,\mu,m}\in H^1(\QQ(m),R_f^*\otimes R_g^*\otimes\cH^\iota)$$ for the Beilinson--Flach element at tame level $m$ associated to the $p$-stabilizations $f^\lambda$ and $g^\mu$.

\begin{theorem}
 Let $h = \max(k_f, k_g)$. Then there exist $\wBF_{\bc,m}\in  \HIw(\QQ(m),R_f^*\otimes R_g^*\otimes\cH^\iota)$, for each $\bc\in\{\#,\flat\}^2$, such that
 \[
 \frac{\fn_{k_f+k_g+2}}{\fn_{h+1}} 
 \begin{pmatrix}
 \BF_{\alpha,\alpha,m}\\ \BF_{\alpha,\beta,m}\\ \BF_{\beta,\alpha,m}\\ \BF_{\beta,\beta,m}
 \end{pmatrix}
 =Q^{-1}\Mlog\cdot 
 \begin{pmatrix}
 \wBF_{\#,\#,m}\\ \wBF_{\#,\flat,m}\\ \wBF_{\flat,\#,m}\\ \wBF_{\flat,\flat,m}
 \end{pmatrix}.
 \] 
\end{theorem}

\begin{proof}
 We fix a basis $\{z_i\}$ of $H^1(\QQ(m),R_f^*\otimes R_g^*\otimes \Lambda^\iota)$ and write $\BF_{\lambda,\mu,m}=\sum F_{\lambda,\mu,i}z_i$ for some $F_{\lambda,\mu,i}\in \cH$. For a fixed $i$, the coefficients $F_{\lambda,\mu,i}$ satisfy the conditions given in Proposition~\ref{prop:vanish} for  $0\le j\le h$ and all $\theta$ of conductor $p^n>1$, thanks to Lemma~\ref{lemma:compareBF} (using case (i) of the lemma for $0 \le j \le \min(k_f, k_g)$, and case (ii) for $\min(k_f, k_g) < j \le h$). Therefore,
 \[
 \frac{\adj(Q^{-1}\Mlog)}{\fn_{k_f+1}\fn_{k_g+1}}\cdot \begin{pmatrix}
 F_{\alpha,\alpha,m}\\F_{\alpha,\beta,m}\\F_{\beta,\alpha,m}\\F_{\beta,\beta,m}
 \end{pmatrix}\in \fn_{h+1}\cH^{\oplus 4}.
 \]
 Hence the result on multiplying $\frac{\fn_{k_f+k_g+2}}{\fn_{h+1}}$ on both sides and the fact that 
 \[
 \frac{\adj(Q^{-1}\Mlog)}{\fn_{k_f+1}\fn_{k_g+1}\fn_{k_f+k_g+2}}
 \]
 is  up to a unit $(Q^{-1}\Mlog)^{-1}$.
\end{proof}

\begin{remark}
 Clearly, if one could show that the coefficients $F_{\lambda, \mu, i}$ also satisfied the conditions of Proposition~\ref{prop:vanish} in the ``anti-geometric'' range $\max(k_f, k_g) + 1 \le j \le k_f + k_g + 1$, then the same argument as above would prove the full strength of Conjecture \ref{conj:signedfactBF}. However, we have not been able to prove this.
\end{remark}

\section{Signed main conjectures}

We shall start this section with the definition of quadruply-signed Selmer groups associated to the Rankin--Selberg product $f\otimes g$. We expect that the quadruply-signed Selmer groups approximate (in an appropriate sense) the Bloch--Kato Selmer groups over finite layers of the cyclotomic tower. Unfortunately, we are unable to present a justification of this expectation. 

We shall formulate our quadruply-signed Iwasawa main conjecture that relates the quadruply-signed Selmer groups to quadruply-signed $p$-adic $L$-functions (which we also define in this section). Assuming the validity of Conjecture~\ref{conj:signedfactBF} (signed-splitting for Beilinson--Flach elements), we shall prove (under mild hypotheses) a divisibility towards the quadruply-signed Iwasawa main conjecture.

\subsection{Quadruply-signed Selmer groups and $p$-adic $L$-functions}

\begin{defn}
 Let $\cS$ denote the set of unordered pairs $\{ (\ts), (\bc)\}$ of ordered pairs, where each of $\ts, \bc$ is one of the symbols $\{\#, \flat\}$, and $(\ts)\neq (\bc)$.
\end{defn}

Note that $\cS$ has 6 elements. We shall define a Selmer group, and formulate a main conjecture, for each $\fS \in \cS$.

\begin{defn}\label{defn:quad}
 Let $\fS = \{ (\ts), (\bc)\} \in \cS$. We define the following objects:
 \begin{itemize}
  \item A compact Selmer group $H^1_{\cF_{\fS}}(\QQ,\TT)$, given by
  $$H^1_{\cF_{\fS}}(\QQ,\TT)\coloneqq \ker\left(H^1(\QQ,\TT)\longrightarrow\frac{H^1(\Qp,\TT)}{\ker\left(\col_{\ts, \Qp}\right) \cap \ker\left(\col_{\bc, \Qp}\right)}\right)\,.$$
  
  \item A discrete Selmer group $\Sel_{\fS}(T^\vee(1)/\QQ(\mu_{p^\infty}))$, given by the kernel of the restriction map
  \[
  H^1(\QQ(\mu_{p^\infty}),T^\vee(1))\lra \prod_{v|p}\frac{H^1(\QQ(\mu_{p^\infty})_v,T^\vee(1))}{H^1_{\fS}(\QQ(\mu_{p^\infty})_v,T^\vee(1))} \times \prod_{v\nmid p}\frac{H^1(\QQ(\mu_{p^\infty})_v,T^\vee(1))}{H^1_f(\QQ(\mu_{p^\infty})_v,T^\vee(1))},
  \]
  where $v$ runs through all primes of $\QQ(\mu_{p^\infty})$, and for $v \mid p$ the local condition $H^1_{\fS}(\QQ(\mu_{p^\infty})_v,T^\vee(1))$ is the orthogonal complement of $\ker\left(\col_{\ts, \Qp}\right)\cap\ker\left(\col_{\bc, \Qp}\right)$ under the local Tate pairing. 
  
  \item Assuming the hypotheses of Proposition \ref{prop:antisymmetryofmatrixofsignedpadicLfunctions}, we define a \emph{quadruply-signed $p$-adic $L$-function} by
  $$\fL_{\fS}\coloneqq \col_{\ts,\Qp}\circ\res_p\left(\BF_{\bc,1}\right) \in \Lambda.$$
 \end{itemize}
\end{defn}

\begin{remark}
 \label{rem:symmetryandantisymmetry}
 The element $\fL_{\fS}$ is only well-defined up to sign, since interchanging the role of $(\ts)$ and $(\bc)$ has the effect of multiplying the $p$-adic $L$-function by $-1$ (this is the content of Proposition  \ref{prop:antisymmetryofmatrixofsignedpadicLfunctions}). However, we shall only be interested in the ideal generated by $\fL_{\fS}$, so the ambiguity of signs is no problem for us.
\end{remark}

%{
%\begin{remark}
%By the weak Leopoldt conjecture and the Poitou--Tate exact sequence (c.f. \cite[\S1.3 and \S A.3]{perrinriou95}), in order for a discrete Selmer group for $T^\vee$ over $\QQ(\mu_{p^\infty})$ to be $\Lambda$-cotorsion, the Selmer condition at $p$ should be of corank at least $2$. This is why we use the kernels of two Coleman maps to define quadruply-signed Selmer groups in Definition~\ref{defn:quad}. We shall give evidence towards the cotorsionness of these Selmer groups in the following section.
%\end{remark} }

\begin{remark}
 We conjecture below an explicit relation between the quadruply-signed Selmer group and the quadruply-signed $p$-adic $L$-function, and offer some partial results towards its validity. For motivational purposes, we shall provide here one philosophical reason why quadruply-signed Selmer group is the correct choice (over doubly-signed Selmer groups). 
 
 Since ${\rm rank}\,\,T^{-} ={\rm rank}\,\,T^{+}=2$, one may deduce using Poitou--Tate global duality (as utilized in the proof of Theorem 5.2.15 and Lemma 5.3.16 of \cite{mr02}) that 
 $${\rm rank}_\LL\, H^1_{\cF_{?}}(\QQ,\TT)-{\rm rank}_\LL\, \Sel_{?}(T^\vee(1)/\QQ(\mu_{p^\infty}))=\begin{cases} 1 & \hbox{ if } ?=(\bullet,\circ) \in \{\#,\flat\}^2\\
 0 &  \hbox{ if } ?=\frak{S}\in \cS\,
 \end{cases}$$
 and one expects, in the spirit of weak Leopoldt conjecture, that ${\rm rank}_\LL\, H^1_{\cF_{?}}(\QQ,\TT)$ should be as small as possible subject to these conditions. Moreover, in line with Bloch-Kato conjectures, one would also expect that ${\rm rank}_\LL\, H^1_{\cF_{?}}(\QQ,\TT)$ is given as the generic order of vanishing of (appropriate linear combinations of) $L$-values associated to the motives $M(f)\otimes M(g)\otimes \chi$, where $\chi$ ranges among Dirichlet characters of $p$-power conductor. In the critical range, note that the generic order of vanishing of  these $L$-values is zero and this should also be the case for at least one of the said linear combinations. This tells us that the corresponding Selmer group ought to have rank zero as well. That is the reason why quadruply-signed Selmer groups are the correct candidates which should relate to the quadruply-signed $p$-adic $L$-functions (that interpolate linear combinations of critical $L$-values) we have defined above.
\end{remark}
%%%%%%%%%%%%%%%%%%%%%%%%%%%%%%%%%%%%%%%%%%%%%%%%%%%%%%%%%%%%%%%%%%%%%%%%%%%%%%%%%%%%%%%%%%%%%%%%%%%%%%%%%%%%%%%%%%%%%%%%%%%%%%%%%%%%%%%%%%%%%%%%%%%%%%%%%%%%%%%%%%%%%%%%%%%%%%%%%%%%%%%%%%%%%%%%%%%%%%%%%%%%%%%%%%%%%%%%%%%%%%%%%%%%%%%%%%%%%%%%%%%%%%%%%%%%%%%%%%%%%%%%%%%%%%%%%%%%%%%%%%%%%%%%%%%%%%%%%%%%%%%%%%%%%%%%%%%%%%%%%%%%%%%%%

\subsection{Quadruply-signed main conjectures}
\label{sect:quadsignedMC}

We are now ready to state the quadruply-signed main conjectures for the Rankin--Selberg convolutions of two $p$-non-ordinary forms. We suppose throughout this section that the hypotheses of Proposition  \ref{prop:antisymmetryofmatrixofsignedpadicLfunctions} are satisfied; in particular, we are assuming  that Conjecture~\ref{conj:signedfactBF} holds. 
\begin{conjecture}
 \label{conj:quadruplysignedmainconj}
 For $\fS=\{(\ts),(\bc)\}\in \cS$ and every character $\eta$ of $\Gamma_{\textup{tor}}$, the $\eta$-isotypic component $e_\eta \Sel_{\fS}(T^\vee(1)/\QQ(\mu_{p^\infty}))$ of the quadruply-signed Selmer group is $\cO[[\Gamma_1]]$-cotorsion and 
 $$\Char_{\cO[[\Gamma_1]]}\left(e_\eta \Sel_{\fS}(T^\vee(1)/\QQ(\mu_{p^\infty}))^\vee\right)\,\, \Big{\vert}\,\,\left(e_\eta\fL_{\fS}\right)$$
 as ideals of $\cO[[\Gamma_1]]$, with equality away from the support of $\coker(\col_{\ts})$ and $\coker(\col_{\bc})$.
\end{conjecture}

\begin{remark}
 \label{rem:quadruplysignedmainconj}
 It is easy to prove that $e_\eta\fL_{\fS}$ is divisible by  $\Char\left(\coker(\col_{\ts})\right)$ in $\Qp\otimes\cO[[\Gamma_1]]$. This {is} the reason why the equality in the asserted divisibility in Conjecture~\ref{conj:quadruplysignedmainconj} excludes the support of $\coker(\col_{\ts})$. As illustrated by our main result (Corollary~\ref{cor:analyticmainconjectureimproved} below) towards the Pottharst-style analytic main conjectures, the error that is accounted by $\coker(\col_{\ts})$ is inessential and it can be recovered.
 
 On the other hand, the reason why we have to avoid {the} support of $\coker(\col_{\bc})$ is more subtle. In view of Proposition~\ref{prop:Descendfromrank2ESproducesrank2ES}, we expect that the Euler system $\{\BF_{\bc,m}\}$ be imprimitive when $\coker(e_\eta\col_{\bc})$ is not the unit ideal, in the sense that the bound it yields on the Selmer group will be off by $\Char\left(\coker(e_\eta\col_{\bc})\right)$. In this case, it is not clear to us whether or not it is possible to improve $\{\BF_{\bc,m}\}$ to a primitive Euler system. These observations are also visible in Corollary~\ref{cor:analyticmainconjectureimproved} below.  
\end{remark}

In the remaining portions of this article we shall present evidence in favour of this conjecture. Until the end, we assume that $|k_f-k_g| \geq 3$ and $\eta=\omega^j$ for a fixed $j$ such that $1+\frac{k_f+k_g}{2}<j\leq \max(k_f, k_g)$.

\begin{proposition}
 \label{prop:someqsignedpadicLisnonzero}
 There exists a choice of $\fS\in \cS$ such that $e_\eta\fL_{\fS}\neq 0$.
\end{proposition}

\begin{proof}
 Let $\cT = \{ \#, \flat\}^2$, and let $\mathscr{M}_{\mathrm{sign}}$ denote the $4 \times 4$ matrix, with rows and columns indexed by $\cT$, whose $(x, y)$ entry is $\col_{x, \Qp}(\BF_{y,1})$. Similarly, let $\cT_{\an} = \{ \alpha_f, \beta_f\} \times \{\alpha_g, \beta_g\}$, and let $\mathscr{M}_{\an}$ denote the $4 \times 4$ matrix whose $x, y$ entry is $\cL_x(\BF_{y, 1})$. By Proposition \ref{prop:antisymmetryofmatrixofsignedpadicLfunctions}, both matrices are anti-symmetric (cf.~Remark~\ref{rem:symmetryandantisymmetry}).
 
 Among the six pairs of non-diagonal entries of $\mathscr{M}_{\an}$, four of them are given by $p$-adic Rankin--Selberg $L$-functions. By Corollary~\ref{cor:nontrivialityofBFelements} and Remark~\ref{rem:nontrivialityofBFelements}, our hypotheses therefore imply that $e_\eta \mathscr{M}_{\an}$ is not the zero matrix.
 
 However, our two matrices are related by the factorisation formula
 \[ \mathscr{M}_{\textup{an}}= (Q^{-1}M_{\log}) \cdot \mathscr{M}_{\mathrm{sign}} \cdot (Q^{-1}M_{\log})\,^T, \]
 so it follows that $e_\eta \mathscr{M}_{\mathrm{sign}}$ is also non-zero. Since the six pairs of non-diagonal entries of $\mathscr{M}_{\mathrm{sign}}$ are exactly the quadruply-signed $p$-adic $L$-functions $\fL_{\fS}$, it follows that at least one of the $e_\eta \fL_{\fS}$ is non-zero as required.
\end{proof}

In order to apply the locally restricted Euler system argument devised in \cite[Appendix A]{kbbleiPLMS}, we will require the validity of the following hypothesis:
\\\\
$\mathbf{(H.nA)}$ Neither $\overline{\rho}_f\vert_{G_{\Qp}}$ nor $\overline{\rho}_f\vert_{G_{\Qp}}\otimes \omega^{-1}$ is isomorphic to $\overline{\rho}_g^\vee\vert_{G_{\Qp}}$, where $\rho_f$ and $\rho_g$ stand for Deligne's (cohomological) representations.

This assumption ensures that $H^0(\Qp,\overline{T})=H^2(\Qp,\overline{T}) = 0$. We will also need to assume that 
\begin{itemize}
 \item[$\mathbf{(BI0)}$] $\varepsilon_f\varepsilon_g$ is non-trivial, $\textup{gcd}(N_f,N_g)=1$ 
\end{itemize}

as well as at least one of the following conditions: 
\begin{itemize}
 \item[$\mathbf{(BI1)}$] Neither $f$ nor $g$ is of CM type, and $g$ has odd weight.
 \item[$\mathbf{(BI2)}$] $f$ is not of CM-type, $g$ is of CM-type and $\varepsilon_g$ is neither the trivial character, nor the quadratic character attached to the CM field.
\end{itemize}
Thanks to \cite{LoefflerGlasgow}, when $\mathbf{(BI0)}$ and either $\mathbf{(BI1)}$ or $\mathbf{(BI2)}$ holds, one may choose a completion of the compositum of the Hecke fields of $f$ and $g$ (and set our coefficient ring $\cO$ to be a finite flat extension of its ring of integers) in a way that the residue characteristic $p$ of $\cO$ is $> k_f+k_g+2$ and the resulting Galois representation $T$ verifies the following ``Big Image" condition that is required to run the Euler system machinery:
\begin{itemize}
 \item The residual representation $\overline{T}$ is absolutely irreducible.
 \item There exists an element $\tau\in \Gal(\overline{\QQ}/\QQ(\mu_{p^\infty}))$ such that $T/(\tau-1)T$ is a free $\cO$-module of rank one.
 \item There exists an element $\sigma\in \Gal(\overline{\QQ}/\QQ(\mu_{p^\infty}))$ which acts on $T$ by multiplication by $-1$.
\end{itemize}

\begin{theorem}
 \label{thm:qsignedmainconj}
 Suppose that $|k_f-k_g| \geq 3$ and $p>k_f + k_g + 2$. Assume the validity of Conjecture~\ref{conj:signedfactBF}, {\bf{(A--Sym)}}, $\mathbf{H.nA}, \mathbf{BI0}$ and at least one of $\mathbf{BI1}-\mathbf{BI2}$. Suppose $j$ is an integer with $1+\frac{k_f+k_g}{2}<j\leq \max(k_f, k_g)$ and choose $\fS$ that verifies the conclusion of Proposition~\ref{prop:someqsignedpadicLisnonzero} with $\eta=\omega^j$. Then the $\omega^j$-isotypic component of the quadruply-signed Selmer group $e_{\omega^j} \Sel_{\fS}(T^\vee(1)/\QQ(\mu_{p^\infty}))$ is $\cO[[\Gamma_1]]$-cotorsion and 
 $$e_{\omega^j}\fL_{\fS}\in \Char_{\cO[[\Gamma_1]]}\left(e_{\omega^j} \Sel_{\fS}(T^\vee(1)/\QQ(\mu_{p^\infty}))^\vee\right)$$
 as ideals of $\cO[[\Gamma_1]]\otimes \Qp$.
\end{theorem}
\begin{remark}
See Proposition~\ref{prop:antisymmetryofmatrixofsignedpadicLfunctions} above where we provide a sufficient condition for the validity of  {\bf{(A--Sym)}}. 
\end{remark}
\begin{proof}[Proof of Theorem~\ref{thm:qsignedmainconj}]
 This is a direct consequence of \cite[Theorem A.14]{kbbleiPLMS}, once we translate the language therein to our set up. Suppose $\fS=\{(\ts),(\bc)\}$. Then the morphism $\Psi$ in loc. cit. corresponds to the map
 $$e_{\omega^j}\col_{\ts}\oplus e_{\omega^j}\col_{\bc}: e_{\omega^j}H^1(\Qp,\TT)\lra \cO[[\Gamma_1]]^{\oplus2}.$$ 
 The so-called $\Psi$-strict Selmer group that is denoted by $H^1_{\cF_\Psi}(-,-)$ in op. cit. corresponds to our quadruply-signed compact Selmer group $e_{\omega^j}H^1_{\cF_\fS}(\QQ,\TT)$ and its dual $H^1_{\cF_\Psi^*}(-,-)$ to our quadruply-signed discrete Selmer group $e_{\omega^j} \Sel_{\fS}(T^\vee(1)/\QQ(\mu_{p^\infty}))$. Moreover, the integer $g$ in loc. cit. equals to $2$ in our case and the expression $\det\left([\Psi(\mathfrak{c_i})_{i=1}^g]\right)$ in the statement of \cite[Theorem A.14(i)]{kbbleiPLMS} is precisely $r_0^{-1}e_{\omega^j}\fL_{\fS}$ in our notation here (where $r_0\in \ZZ$ is as in the formulation of Conjecture~\ref{conj:signedfactBF}; note that since we have inverted $p$, this quantity does not make an appearance in the statement of our theorem). 
 
 Our running hypotheses guarantee the validity of all required assumptions for this result; only checking the validity of the condition (H.V) of \cite[Appendix A]{kbbleiPLMS} (which translates in our setting to the condition that the quadruply-signed compact Selmer group $e_{\omega^j}H^1_{\cF_\fS}(\QQ,\TT)$ be trivial) requires some work. The remainder of this proof is dedicated to show that the hypothesis (H.V) holds true in our set up.
 
 We start with the reformulation of this condition. The Selmer group denoted by $H^1_{\cF_L}(-,-)$ in  \cite[Appendix A]{kbbleiPLMS} corresponds to our Selmer group $e_{\omega^j}H^1_{\cFbc}(\QQ,\TT)$ and the condition (H.V) requires that the map
 $$\xymatrix{e_{\omega^j}H^1_{\cFbc}(\QQ,\TT)\ar[rrrr]^(.55){\left(e_{\omega^j}\col_{\ts}\,\oplus\, e_{\omega^j}\col_{\bc}\right)\circ\res_p}&&&&\cO[[\Gamma_1]]^{\oplus 2}}$$
 be injective. By the defining property of $H^1_{\cFbc}(\QQ,\TT)$, this is equivalent to checking that the map 
 $$\xymatrix{e_{\omega^j}H^1_{\cFbc}(\QQ,\TT)\ar[rrr]^(.55){e_{\omega^j}\col_{\ts}\circ \res_p}&&&\cO[[\Gamma_1]]}$$
 is injective. Since we already have $\BF_{\bc,1} \in H^1_{\cFbc}(\QQ,\TT)$ (as given by Conjecture~\ref{conj:signedfactBF}, which we assume to hold)  and we know (thanks to our choice of $\fS$) that 
 $$e_{\omega^j}\col_{\ts}\circ\res_p\left(e_{\omega^j}\BF_{\bc,1}\right)=e_{\omega^j}\fL_\fS\neq 0,$$
 the condition (H.V) is equivalent to the requirement that $e_{\omega^j}H^1_{\cFbc}(\QQ,\TT)$ has rank one. By Poitou--Tate global duality, this is in turn equivalent to checking that $e_{\omega^j}\Sel_{\bc}(T^\vee(1)/\QQ(\mu_{p^\infty}))$ is $\cO[[\Gamma_1]]$-cotorsion. We shall explain how to verify this fact.
 
 Let us set $\TT_1\coloneqq T\otimes\cO[[\Gamma_1]]^\iota$ (with diagonal Galois action). Choose a degree one polynomial $l \in \cO[[\Gamma_1]]$ that does not divide $e_{\omega^j}\fL_{\fS}\cdot \Char\left(\coker\left(\col_{\bc}\right)\right)$ and define $X\coloneqq \TT_1\otimes \omega^{-j}/l$. Observe that we have (for $F=\Qp$ or any finite abelian extension of $\QQ$ unramified at $p$)
 \begin{equation}
  \label{eqn:inflationrestrictionshapiro}
  H^1(F,\TT_1\otimes\omega^{-j})\stackrel{\sim}{\lra} e_{{\omega^{j}}}H^1_\Iw(F(\mu_p),T)
 \end{equation}
 by the inflation-restriction sequence. We define 
 $$H^1_{\cFbc}(\Qp,\TT_1\otimes\omega^{-j})\subset H^1(\Qp,\TT_1\otimes\omega^{-j})$$ 
 as the submodule that gets mapped isomorphically onto $e_{\omega^j}H^1_{\cFbc}(\QQ,\TT)$ under the map \eqref{eqn:inflationrestrictionshapiro}. The isomorphism together with $e_{\omega^j}\col_{\bc}$ also induces a map 
 $$H^1(\Qp,\TT_1\otimes\omega^{-j})\lra \cO[[\Gamma_1]]$$
 (which we shall denote by the same symbol), whose kernel is precisely the submodule $H^1_{\cFbc}(\Qp,\TT_1\otimes\omega^{-j})$. 
 
 For primes $\ell\neq p$, we shall also set $H^1_{\cFbc}(\QQ_\ell,\TT_1\otimes\omega^{-j})\coloneqq H^1(\QQ_\ell,\TT_1\otimes\omega^{-j})$ so that $\cFbc$ is a Selmer structure on $\TT_1\otimes\omega^{-1}$ in the sense of \cite{mr02}. It is easy to see that the dual Selmer group $H^1_{\cFbc^*}(\QQ,\TT_1^\vee(1)\otimes\omega^{j})^\vee$ is $\cO[[\Gamma_1]]$-torsion if and only if $e_{\omega^j}\Sel_{\bc}(T^\vee(1)/\QQ(\mu_{p^\infty}))^\vee$ is $\cO[[\Gamma_1]]$-torsion. Thanks to \cite[Lemma 3.5.3]{mr02}, our claim that $e_{\omega^j}\Sel_{\bc}(T^\vee(1)/\QQ(\mu_{p^\infty}))$ is cotorsion follows once we verify that 
 $$H^1_{\cFbc^*}(\QQ,\TT_1^\vee(1)\otimes\omega^{j})[l]\cong H^1_{\cFbc^*}(\QQ,X^\vee(1))$$ 
 has finite cardinality. Note that we have written $\cFbc$ (resp., $\cFbc^*$) for the propagation of the Selmer structure $\cFbc$ to $X$ (resp., for the Selmer structure on $X^\vee(1)$ dual to $\cFbc$ on $X$).
 
 Consider the following diagram with exact rows and Cartesian squares:
 $$\xymatrix{0 \ar[r] &H^1_{\cFbc}(\Qp,\TT_1\otimes\omega^{-j})\ar[r]\ar@{-->}[d]&H^1(\Qp,\TT_1\otimes\omega^{-j})\ar[r]^(.65){\col_{\bc}}\ar@{->>}[d]&\cO[[\Gamma_1]]\ar@{->>}[d]\\
  0 \ar[r] &H^1_{\cF}(\Qp,X)\ar[r]&H^1(\Qp,X)\ar[r]^(.6){\phi}&\cO
 }$$
 where the vertical arrows are induced by reduction modulo $l$ (which we henceforth denote by $\pi_X$); the map $\phi$ is defined by the Cartesian square on the right; the submodule $H^1_{\cF}(\Qp,X)$ by the exactness of the second raw and the dotted arrow by chasing the diagram. Note that the map $\phi$ is not the zero map thanks to our choice of $l$.
 
 For $\ell\neq p$, let us also define 
 $$H^1_{\cF}(\QQ_\ell,X)\coloneqq \ker\left(H^1(\QQ_\ell,X)\lra H^1(\QQ_\ell^\textup{ur},X\otimes\Qp)\right).$$ 
 It follows from \cite[Lemma 5.3.1(i)]{mr02} (together with the dotted arrow in the diagram above) that the reduction map modulo $l$ induces a map
 \begin{equation}
  H^1_{\cFbc}(\QQ_\ell,\TT_1\otimes\omega^{-j})\lra H^1_{\cF}(\QQ_\ell,X)
 \end{equation}
 for every prime $\ell$ (including $p$), which factors through the injection $H^1_{\cFbc}(\QQ_\ell,X)\subset H^1_{\cF}(\QQ_\ell,X)$. In particular, we have $\cFbc\leq \cF$ in the sense of \cite[Definition 2.1.1]{mr02} and we have an injective map 
 \begin{equation}
  \label{eqn:injectingoneKStoanother}
  \mathrm{KS}(X,\cFbc)\hookrightarrow \mathrm{KS}(X,\cF)
 \end{equation}
 between the corresponding modules of Kolyvagin systems.
 
 Using Lemma 3.7.1 of op. cit., it follows that the Selmer structure $\cF$ is cartesian (in the sense of \cite[Definition 1.1.4]{mr02}). Moreover, it is easy to see (recalling that the map $\phi$ is not the zero map) that the core Selmer rank of $\cF$ equals one. In particular, by \cite[Corollary 5.2.13]{mr02}, the finiteness of $H^1_{\cF}(\QQ_\ell,X)$ is equivalent to exhibiting a single Kolyvagin system in $\mathrm{KS}(X,\cF)$, whose initial term is non-zero. We shall prove that the Euler system $\{e_{\omega^j}\BF_{\bc,n}\}$ of doubly-signed Beilinson--Flach classes descend to a Kolyvagin system with this property and this completes the proof.

 Let us write $\BF_{\bc}^{\omega^j} \in H^1(\QQ(\mu_n),\TT_1\otimes\omega^{-j})$ for the class that maps to $e_{\omega^j} \BF_{\bc,n}$ under the isomorphism \eqref{eqn:inflationrestrictionshapiro} and let us set $\BF_{\bc}^{\omega^j}\coloneqq \{\BF_{\bc,n}^{\omega^j} \}$. As we have explained in Remark~\ref{rmk:locrestrict}, $\left\{\BF_{\bc}^{\omega^j}\right\}$ is a locally restricted Euler system. Moreover, \cite[Theorem A.11]{kbbleiPLMS} applies (with our choices here, recall that the Selmer structure $\cF_L$ in loc. cit. corresponds to our $\cFbc$) and produces a Kolyvagin system 
 $${\kappa}^{\bc}=\{\kappa_n^{\bc}\}\in \mathrm{KS}(\TT_1\otimes\omega^{-j},\cFbc).$$ 
 We remark that the Kolyvagin system ${\kappa}^{\bc}$ is obtained from the Euler system $\BF_{\bc,n}^{\omega^j}$ via \cite[Theorem 5.3.3]{mr02}. However, the results of Mazur and Rubin show a priori  only that ${\kappa}^{\bc} \in  \mathrm{KS}(\TT_1\otimes\omega^{-j},\cF_\LL)$ (namely, ${\kappa}^{\bc}$ is a Kolyvagin system for the canonical Selmer structure $\cF_\LL$ of \cite[Definition 5.3.2]{mr02}). The fact that the classes ${\kappa}^{\bc}_n$ verify the required local conditions at $p$ follows from the fact that $\BF_{\bc}^{\omega^j}$ is in fact locally restricted, as explained in detail in the proof of \cite[Theorem A.11]{kbbleiPLMS}.

 On projecting ${\kappa}^{\bc}$ via $\pi_X$ and composing with the injection \eqref{eqn:injectingoneKStoanother}, we obtain a Kolyvagin system $\pi_X\left({\kappa}^{\bc}\right)\in  \mathrm{KS}(X,\cF)$. We have 
 $$\pi_X\left({\kappa}^{\bc}\right)_1=\pi_X\left({\kappa}^{\bc}_1\right)=\pi_X\left(\BF_{\bc,1}^{\omega^j}\right)$$
 for the initial term of this Kolyvagin system. We are reduced to prove that $\pi_X\left(\BF_{\bc,1}^{\omega^j}\right)\neq 0$ for the choices above.
 
 We now recall that $l \nmid e_{\omega^j}\fL_{\fS}=\col_{\ts}\circ\res_p\left(\BF_{\bc,1}^{\omega^j}\right)$ by the choice we made on the degree one polynomial $l$. This in turn implies that $\pi_X\left(\BF_{\bc,1}^{\omega^j}\right)\neq 0$, as required.
\end{proof}

\begin{remark}
In the proof of Theorem~\ref{thm:qsignedmainconj}, we only needed the anti-symmetry property in {\bf(A--Sym)} to hold for  $\{(\ts),(\bc)\}$  with $m=1$ and  $\{(\bc),(\bc)\}$ for all $m\in \mathcal{N}(\mathcal{P})$.
\end{remark}

%%%%%%%%%%%%%%%%%%%%%%%%%%%%%%%%%%%%%%%%%%%%%%%%%%%%%%%%%%%%%%%%%%%%%%%%%%%%%%%%%%%%%%%%%%%%%%%%%%%%%%%%%%%%%%%%%%%%%%%%%%%%%%%%%%%%%%%%%%%%%%%%%%%%%%%%%%%%%%%%%%%%%%%%%%%%%%%%%%%%%%%%%%%%%%%%%%%%%%%%%%%%%%%%%%%%%%%%%%%%

\section{Analytic main conjectures}

Our goal in this section is to translate our results on signed Iwasawa main conjectures into the ``analytic'' language of Pottharst and Benois (see \cite{jaycyclo, jayanalyticselmer,benoisSelmerComplexes}). This gives main conjectures which directly involve the $p$-adic Rankin--Selberg $L$-functions \eqref{eq:fourLp}, which is advantageous since the interpolating properties of these $L$-functions are much more explicit than those of the signed $p$-adic $L$-functions $\mathfrak{L}_{\fS}$ of the previous section. However, it has the disadvantage of throwing away all $p$-torsion information.

\subsection{Cohomology of $(\varphi, \Gamma)$-modules}
\label{sect:phigammacoho}

For each $ 0 \leq r < 1$, let 
\[ \mathrm{ann(r, 1)} \coloneqq \lbrace x \in \mathbb{C}_p : r \leq |x|_p < 1 \rbrace. \]
For $E$ the finite extension of $\Qp$ in \S 1, we define the \emph{Robba ring}
\[ 
\mathcal{R}_E \coloneqq \left\lbrace f(\pi) = \sum_{n = -\infty}^{\infty} a_n\pi^{n} \in E[[\pi]] \ \middle|\ \begin{array}{c} \text{$f(\pi)$ converges on $\mathrm{ann}(r, 1)$} \\ \text{for some $r$} \end{array} \right \rbrace. 
\]
The Robba ring comes equipped with actions of $\Gamma$ and the Frobenius $\varphi$, via the same formulae as in \S\ref{S:review} above. 

\begin{defn}
 A $(\varphi, \Gamma)$-module over $\mathcal{R}_E$ is a free module of finite rank $d$ endowed with a semi-linear Frobenius $\varphi$ such that $\mathrm{Mat}(\varphi) \in \mathrm{GL}_d(\mathcal{R}_E)$ and with a continuous commuting semi-linear action of $\Gamma$.   
\end{defn}

There exists a functor $V \mapsto \DD_{\rig}^{\dagger}(V)$ between the category of $p$-adic representations of $G_p \coloneqq \mathrm{Gal}(\overline{\QQ}_p/\Qp)$ with coefficients in $E$ and the category of $(\varphi, \Gamma)$-modules over $\mathcal{R}_E$ (see \cite{cherbonniercolmez}, \cite{fontaine90} and \cite{berger02}). 

For any $(\varphi, \Gamma)$-module $\DD$, we define its \emph{analytic Iwasawa cohomology} $H^{\bullet{}}_{\an}(\DD)$ to be the cohomology of the complex
\[ \big[ \DD \xrightarrow{\psi - 1} \DD \big], \]
where $\psi$ is the left inverse of $\varphi$ and the terms are placed in degrees 1 and 2 respectively. If $V$ is a $p$-adic representation of $G_p$, then the results of \cite{jayanalyticselmer} show that 
\[ H^i_{\an}(\Qp(\mu_{p^{\infty}}), V) \cong  H^i_{\an}(\DD^{\dag}_{\rig}(V)) \]
where $ H^i_{\an}(\Qp(\mu_{p^{\infty}}), V) = H^i(\Qp, \VV^{\dag})$ denotes the analytic Iwasawa cohomology of $V$ as defined in \S \ref{subsec:iwasawacohomology}.

\subsection{Selmer complexes}
\label{subsec_SelmerComplexAnalytic}
Now let $V$ be a continuous $E$-linear representation of $G_\QQ$ which is unramified at almost all primes. Following Benois and Pottharst, we now recall the complexes defined in \S 2.3 of \cite{benoisSelmerComplexes}, which are are an ``analytic'' version of the Selmer complexes of Nekov\'a\v{r} in \cite{nekovarselmer}. Letting $\DD$ be any $(\varphi, \Gamma)$-submodule of $\DD^\dag_{\rig}(V |_{G_{\Qp}})$, we obtain a \emph{Selmer complex} $S(\QQ, \VV^\dag; \DD)$ in the derived category of $\cH$-modules. This is defined by a mapping fibre
\[ 
\operatorname{cone} 
\Bigg[ 
C^{\bullet{}}(G_{\QQ, \Sigma}, \VV^{\dag}) \oplus \bigoplus_{v \in \Sigma} U_v^{+}
\xrightarrow{\mathrm{res}_v - i_v^{+}}\bigoplus_{v \in \Sigma}C^\bullet(G_{\QQ_v}, \VV^{\dag})\Bigg] [ - 1 ]
\]
where $\Sigma$ is any sufficiently large finite set of places, and the $U_v^+$ are appropriate ``local condition'' complexes; we choose these to be the unramified local conditions for $v \ne p$, and the local condition at $p$ is given by the analytic Iwasawa cohomology of the submodule $\DD \subseteq \DD^\dag_{\rig}(V |_{G_{\Qp}})$. We write $H^\bullet(\QQ, \VV^\dag; \DD)$ for the cohomology groups of the Selmer complex.

We shall apply this with $V = R_f^* \otimes R_g^*$ and with local conditions $\DD$ given by $\varphi$-stable subspaces of $\Dcris(V)$. We set
\[
\begin{array}{lcl}
\Dcris(V)^{\lambda} & \coloneqq &   \Dcris(V)^{\lambda_f\lambda_g} \oplus  \Dcris(V)^{\lambda_f\mu_g} \\
& = & \Dcris(R_f^{*})^{\lambda_f} \otimes \Dcris(R_g^{*})
\end{array} 
\]
and
\[
\Dcris(V)^{\lambda, \mu} \coloneqq   \Dcris(R_f^{*})^{\lambda_f} \otimes \Dcris(R_g^{*})  + \Dcris(R_f^{*}) \otimes \Dcris(R_g^{*})^{\mu_g}. 
\] 
Let $\DD_{\lambda, \mu}$ and $\DD_{\lambda}$ be the $(\varphi, \Gamma)$-submodules of $\DD^{\dag}_{\rig}(T)$ corresponding to $ \Dcris(T)^{\lambda}$ and $\Dcris(T)^{\lambda, \mu}$ respectively (see \cite{berger08} for an explicit description).  We note that $\DD_{\lambda, \mu}$ and $\DD_{\lambda}$ play the role of \emph{regular submodules} defined by Benois in \cite{benoisSelmerComplexes}.

\subsection{Analytic main conjectures}

We are now ready to state the \emph{analytic Iwasawa main conjecture} formulated by Benois and Pottharst, specialized to our setting. 

\begin{conjecture} 
 \label{conj:analyticconjecture}
 The module $H^2(\QQ, \VV^{\dag}; \DD_\lambda)$ is torsion, and its characteristic ideal is given by
 \[ 
 \operatorname{char}_{\cH}H^2(\QQ, \VV^{\dag}; \DD_\lambda) = 
 L_p(f_{\lambda}, g)\cdot\cH,
 \]
 where $L_p(f_{\lambda}, g)$ denotes the geometric $p$-adic $L$-function attached to $f$ and $g$.
\end{conjecture}

(One can similarly define a Selmer group and formulate a main conjecture for the ``extra'' $p$-adic $L$-functions $L_p^?$ introduced above, but we shall not give the details here.) We now proceed to show how our results in the previous sections on signed Selmer groups aid us to deduce some partial results towards Conjecture \ref{conj:analyticconjecture}.

\subsection{Bounds for analytic Selmer groups}

Throughout the remainder of this section, the hypotheses of Theorem~\ref{thm:qsignedmainconj} are in effect. We begin by reformulating the result of Theorem~\ref{thm:qsignedmainconj} as a bound for $H^2(\QQ, \TT)$; we shall then translate this into a bound for the analytic Selmer group.

\begin{proposition}
 There exists a choice $(\bc)$ and $(\ts)$ such that neither $e_\eta \col_{\ts}\circ\res_p\left(\BF_{\bc,1}\right)$ nor $e_\eta \col_{\bc}\circ\res_p\left(\BF_{\ts,1}\right)$ is zero.
\end{proposition}

\begin{proof}
 This follows from Propositions~\ref{prop:someqsignedpadicLisnonzero} and~\ref{prop:antisymmetryofmatrixofsignedpadicLfunctions}.
\end{proof}

From now on, we fix $(\bc)$ and $(\ts)$  such that $e_\eta \col_{\ts}\circ\res_p\left(\BF_{\bc,1}\right)\neq 0$.

As in Proposition \ref{prop:Hproj}, this implies that $e_\eta H^1(\QQ, \TT)$ is free of rank two. Choose an $e_\eta \Lambda$-basis $\{c_1, c_2\}$ of this module. 

\begin{defn}
 
 Let $r_1,r_2,D \in \LL$ be non-zero elements such that 
 $$E\cdot\BF_{\bc,1}=r_1\left(\col_{\bc}\circ\res_p\left(c_1\right)c_2-\col_{\bc}\circ\res_p\left(c_2\right)c_1\right)$$
 $$E\cdot\BF_{\ts,1}=r_2\left(\col_{\ts}\circ\res_p\left(c_1\right)c_2-\col_{\ts}\circ\res_p\left(c_2\right)c_1\right).$$
 Here, the first equality takes place in $H^1_{\cFbc}(\QQ,\TT)$, whereas the second in $H^1_{\cF_{\ts}}(\QQ,\TT)$.
\end{defn}
Note {that} $E, r_1,r_2$ with the required properties exist since both modules $H^1_{\cFbc}(\QQ,\TT)$ and $H^1_{\cF_{\ts}}(\QQ,\TT)$ have rank one.
\begin{lemma}
 $r_1=-r_2$.
\end{lemma}
\begin{proof}
 This is an immediate consequence of Proposition~\ref{prop:antisymmetryofmatrixofsignedpadicLfunctions}.
\end{proof}

\begin{proposition}
 \label{prop:signedmainconjtoboundsonH2}
 $E\cdot\Char\left(H^2(\QQ,\TT)\right)$ divides $r_1\cdot \Char\left( \coker\left(\col_{\bc}\right)\right)$.
\end{proposition}
\begin{proof}
 Let us set $H^1_{/\bc}(\Qp,\TT) \coloneqq H^1(\Qp,\TT)/H^1_{\cFbc}(\Qp,\TT)$ and write $\res^s_{\bc}$ for the compositum of the arrows
 $$H^1(\QQ,\TT)\stackrel{\res_p}{\lra} H^1(\Qp,\TT)\lra H^1_{/\bc}(\Qp,\TT).$$ 
 Poitou--Tate global duality gives rise to the following five-term exact sequence:
 \begin{align}
  \notag 0\lra H^1_{\cFbc}(\QQ,\TT)/\LL\cdot\BF_{\bc,1}\lra H^1(\QQ,&\TT)/(\BF_{\bc,1},\BF_{\ts,1})\lra \frac{H^1_{/\bc}(\Qp,\TT)}{\res^s_{\bc}(\BF_{\ts,1})}\\
  \label{eqn:5termsignedexactsequence}&\lra H^1_{\cFbc^*}(\QQ,\TT^\vee(1))^\vee\lra H^2(\QQ,\TT)\lra 0\,.
 \end{align}
 
 The locally restricted Euler system machinery shows that
 \begin{equation}
  \label{eqn:ESboundwithbulletcircSelmerstructure}
  \Char\left(H^1_{\cFbc^*}\left(\QQ,\TT^\vee(1)\right)^\vee\right) \mid \Char\left(H^1_{\cFbc}(\QQ,\TT)/\LL\cdot\BF_{\bc,1}\right).
 \end{equation}
 Combining \eqref{eqn:5termsignedexactsequence} with the divisibility \eqref{eqn:ESboundwithbulletcircSelmerstructure}, we infer that 
 \begin{align}
  \label{eqn:boundonH2step1}
  \Char\left(H^2(\QQ,\TT)\right)\Char\left(\coker\left(\col_{\bc}\right)\right)^{-1}&\col_{\bc}\circ\res_p\left(\BF_{\ts,1}\right) \\
  \notag&\hbox{divides } \Char\left(\frac{H^1(\QQ,\TT)}{(\BF_{\bc,1},\BF_{\ts,1})}\right).
 \end{align}
 Moreover, we have
 \begin{equation}
  \label{eqn:signedcolsignedbasasigneddet}
  \col_{\bc}\left(\BF_{\ts,1}\right)=E^{-1}r_2 \det\circ\, \col(\ts,\bc)\end{equation}
 where we have set 
 $$\det\circ\, \col(\ts,\bc)\coloneqq\det\left(\begin{array}{cc} \col_{\ts}\circ\res_p\left(c_1\right)&\col_{\bc}\circ\res_p\left(c_1\right)\\
 \col_{\ts}\circ\res_p\left(c_2\right)&\col_{\bc}\circ\res_p\left(c_2\right)
 \end{array}\right).$$ 
 Let $f_{\ts}\coloneqq\col_{\ts}\circ\res_p$ (and we similarly define $f_{\bc}$). Note that since we have
 \begin{align}\notag\Char\,({E(c_1,c_2)}\Big{/}(r_1\left(f_{\bc}\left(c_1\right)c_2-f_{\bc}\left(c_2\right)c_1\right), \,&r_2\left(f_{\ts}\left(c_1\right)c_2-f_{\ts}\left(c_2\right)c_1\right)))\\
  &=E^{-2}r_1r_2\det\circ\, \col(\ts,\bc)\,.
 \end{align}
 it follows from definitions that
 \begin{equation}
  \label{eqn:boundonH2step2}
  \Char\left(\frac{H^1(\QQ,\TT)}{(\BF_{\bc,1},\BF_{\ts,1})}\right)=E^{-2}r_1r_2\det\circ\, \col(\ts,\bc).
 \end{equation}
 Combining \eqref{eqn:boundonH2step1}, \eqref{eqn:signedcolsignedbasasigneddet} and \eqref{eqn:boundonH2step2}, the asserted divisibility follows.
\end{proof}
Let us choose $F,s_1,s_2 \in \cH\setminus\{0\}$ so that
$$F\cdot\BF_{\lambda,\mu}=s_1\left(\cL_{\lambda,\mu}\circ\res_p\left(c_1\right)c_2-\cL_{\lambda,\mu}\circ\res_p\left(c_2\right)c_1\right)$$
$$F\cdot\BF_{\lambda,\mu^\prime}=s_2\left(\cL_{\lambda,\mu^\prime}\circ\res_p\left(c_1\right)c_2-\cL_{\lambda,\mu^\prime}\circ\res_p\left(c_2\right)c_1\right)\,.$$
We also set 
$$\det\circ\,\cL(\lambda,\mu,\lambda,\mu^\prime)\coloneqq\det\left(\begin{array}{cc} \cL_{\lambda,\mu}\circ\res_p\left(c_1\right)&\cL_{\lambda,\mu^\prime}\circ\res_p\left(c_1\right)\\
\cL_{\lambda,\mu}\circ\res_p\left(c_2\right)&\cL_{\lambda,\mu^\prime}\circ\res_p\left(c_2\right)
\end{array}\right).$$

\begin{lemma}
 $s_1=-s_2$.
\end{lemma}
\begin{proof}
 This follows from Theorem~\ref{thm:liftingBFattrivialtamelevelinFracH}.
\end{proof}
\begin{proposition}
 
 We have the following divisibility of $\cH$-ideals:
 \[ \left. \Char\left( H^2(\QQ, \VV^\dagger; \DD_{\lambda,\mu})\right)
 \ \middle|\  
 \frac{r_1 F}{s_1 E} \cdot  \Char\left(\coker \col_{\bc}\right)\cdot \Char\left(\frac{H^1(\QQ, \VV^\dagger; \DD_{\lambda,\mu})}{\cH\cdot \BF_{\lambda,\mu,1}}\right)\right. .\]
\end{proposition}
\begin{proof}
 The proof of this assertion is essentially the proof of Proposition~\ref{prop:signedmainconjtoboundsonH2} in reverse, starting off with the 5-term exact sequence of $\cH$-modules
 \begin{multline*}
  0 \lra \frac{H^1(\QQ, \VV^\dagger; \DD_{\lambda,\mu})}{\cH \cdot \BF_{\lambda,\mu,1}} \lra
  \frac{ H^1(\QQ,\VV^\dagger)} 
  {\cH \cdot \BF_{\lambda,\mu,1} + \cH \cdot \BF_{\lambda,\mu^\prime,1}} \lra
  \frac{H^1_{/\lambda,\mu}(\Qp,\VV^\dagger)}{\res_p^s\left(\BF_{\lambda,\mu^\prime}\right)}\\
  \lra H^2(\QQ, \VV^\dagger; \DD_{\lambda,\mu})\lra H^2(\QQ,\VV^\dagger) \lra  0,
 \end{multline*}
 where $H^1_{/\lambda,\mu}(\Qp,\VV^\dagger)\coloneqq H^1(\Qp,\VV^\dagger)/H^1_{\mathrm{an}}(\Qp, \DD_{\lambda,\mu})$ and $\res_p^s$ is the natural map
 $$\res_p^s: H^1(\QQ,\VV)\lra H^1_{/\lambda,\mu}(\Qp,\VV^\dagger).$$
 Notice that we also rely on the fact that $\cL_{\lambda,\mu}: H^1_{/\lambda,\mu}(\Qp, \VV^\dagger) \to \cH$ is surjective, and the fact that $H^2(\QQ,\VV^\dagger) = \cH \otimes_{\Lambda} H^2(\QQ, \TT)$ when we use here the statement of Proposition~\ref{prop:signedmainconjtoboundsonH2}.
\end{proof}

\begin{corollary}
 \label{cor:boundonthetrueSelmer}
 The ideal $\Char\left(H^2(\QQ, \VV^\dagger; \DD_{\lambda})\right)$ divides the ideal
 $$\Char\left(\coker \col_{\bc} \right)E^{-1}r_1Fs_1^{-1}\cL_{\lambda,\mu^\prime}(\BF_{\lambda,\mu,1})\,.$$
\end{corollary}

%%%%%%%%%%%%%%%%%%%%%%%%%%%%%%%%%%%%%%%%%%%%%%%%%%%%%%%%%%%%%%%%%%%%%%%%%%%%%%%%%%%%%%%%%%%%%%%%%%%%%%%%%%%%%%%%%%%%%%%%%%%%%%%%%%%%%%%%%%%%%%%%%%%%%%%%%%%%%%%%%%%%%%%%%%%%%%%%%%%%%%%%%%%%%%%%%%%%%%%%%%%%%%%%%%%%%%%%%%%%

\begin{proposition}
 \label{prop:rFequalssE}
 $r_1F =r_0s_1E$.
\end{proposition}
\begin{proof}
 Recall the element $\BF_1 \in \operatorname{Frac} \cH \otimes_{\Lambda} \bigwedge^2 H^1_{\Iw}(\QQ(\mu_{p^\infty}), T)$ given as in Theorem~\ref{thm:liftingBFattrivialtamelevelinFracH}. Let us choose $h \in {\rm Frac}(\cH)^\times$ so that we have $e_\eta\BF_1  =h\cdot (c_1 \wedge c_2).$
 It follows from \eqref{eqn:ESrank2impliessignedfactorizationforBF} together with the defining property of $E$ and $r_1$ that  $r_0h=r_1/E$. Comparing the definition of $h$ and $s_1/F$, we also see that $h=s_1/F$. Our assertion follows.
\end{proof}
\begin{corollary}
 \label{cor:analyticmainconjectureimproved}
 The ideal $\Char\left(H^2(\QQ, \VV^\dagger; \DD_{\lambda})\right)$ divides 
 $$\Char\left(\coker \col_{\bc}\right) L_p(f_\lambda, g).$$
\end{corollary}
\begin{proof}
 This follows on combining Corollary~\ref{cor:boundonthetrueSelmer} and Proposition~\ref{prop:rFequalssE}, together with the observation that $r_0 \in \ZZ$ is invertible in the ring $\cH$.
\end{proof}

%%%%%%%%%%%%%%%%%%%%%%%%%%%%%%%%%%%%%%%%%%%%%%%%%%%%%%%%%%%%%%%%%%%%%%%%%%%%%%%%%%%%%%%%%%%%%%%%%%%%%%%%%%%%%%%%%%%%%%%%%%%%%%%%%%%%%%%%%%%%%%%%%%%%%%%%%%%%%%%%%%%%%%%%%%%%%%%%%%%%%%%%%%%%%%%%%%%%%%%%%%%%%%%%%%%%%%%%%%%%

\appendix

\section{Images of Coleman maps}

\renewcommand{\thetheorem}{A.\arabic{theorem}}

We describe the images of various Coleman maps up to pseudo-isomorphisms, which is relevant to our discussion in Remark~\ref{rem:quadruplysignedmainconj}.  Throughout, $T$ denotes the representation $R_f^*\otimes R_g^*$ and $\cL_T:\HIw(\Qp,T)\rightarrow\cH\otimes \Dcris(T)$ as before. We recall the following result from \cite{leiloefflerzerbes11}.

\begin{proposition}\label{prop:imageL}
 Let $z\in \HIw(\Qp,T)$, $\delta$ a Dirichlet character of conductor $p^n>1$ and $0\le j\le k_f+k_g+1$, then
 \begin{align*}
  &(1-p^j\vp)^{-1}(1-p^{-j-1}\vp^{-1})\chi^j(\cL_T(z))\in \Fil^{-j}\Dcris(T),\\
  &(p^j\vp)^{-n}(\chi^j\delta(\cL_T(z)))\in\Qp(\mu_{p^{n-1}})\otimes \Fil^{-j}\Dcris(T).
 \end{align*}
\end{proposition}
\begin{proof}
 This is \cite[Proposition~4.8]{leiloefflerzerbes11}.
\end{proof}
Given a character $\eta\in\hat\Delta$ and an integer $0\le j\le k_f+k_g+1$, we define
\[
V_{\eta,j}\coloneqq \begin{cases}
(1-p^j\vp)(\vp-p^{-j-1})^{-1}\Fil^{-j}\Dcris(T)&\text{if $\eta=\omega^{j}$,}
\\
\Fil^{-j}\Dcris(T)&\text{otherwise.}
\end{cases}
\]
Via our chosen basis $\{v_1,v_2,v_3,v_4\}$ of $\Dcris(T)$, we identify $\Dcris(T)\otimes \Qp$ with $L$.
Let us write $\uCol:\HIw(\Qp,T)\to \cO\otimes\Lambda^{\oplus 4}$ for the morphism given by $(\col_{\#,\#},\col_{\#,\flat},\col_{\flat,\#},\col_{\flat,\flat})$. For each $\eta\in\hat\Delta$, we may then identify $e_\eta\uCol$ as a map landing inside $\cO[[X]]^{\oplus 4}$ as before via $1+X=\gamma$, where $\gamma$ is our chosen topological generator of $\Gamma_1$. Recall that $u=\chi(\gamma)$.

\begin{corollary}
 Let $z\in \HIw(\Qp,T)$, $0\le j \le k_f+k_g+1$ and $\eta\in\hat\Delta$. Then, 
 \[
 e_\eta\uCol(z)|_{X=u^j-1}\in V_{\eta,j}.
 \]
\end{corollary}
\begin{proof}
 Note that  $e_\eta\cL_T(z)|_{X=u^j-1}=\chi^j(\eta\omega^{-j})(\cL_T(z))$.
 Therefore, Proposition~\ref{prop:imageL} says that
 \[
 e_\eta\cL_T(z)|_{X=u^j-1}\in P_{\eta,j}(\Fil^{-j}\Dcris(T)),
 \]
 where $P_{\eta,j}$ is given by
 \[
 \begin{cases}
 (1-p^j\vp)(1-p^{-j-1}\vp^{-1})^{-1}&\text{if $\eta=\omega^{j}$,}
 \\
 \vp&\text{otherwise.}
 \end{cases}
 \]
 Recall from \eqref{eq:decompPR} that
 \[
 \cL_T(z)=\begin{pmatrix}
 v_1&v_2&v_3&v_4
 \end{pmatrix}\cdot \Mlog\cdot\begin{pmatrix}
 \col_{\#,\#}(z)\\
 \col_{ \#,\flat}(z)\\
 \col_{ \flat,\#}(z)\\
 \col_{ \flat,\flat}(z)
 \end{pmatrix}.
 \]
 Note that $e_\eta\Mlog|_{X=u^j-1}=A$, which is a consequence of \eqref{eq:congruentid}. Recall that $A$ is the matrix of $\vp$ with respect to the basis $\{v_1,v_2,v_3,v_4\}$, which implies our result.
\end{proof}

Following \cite[Proposition~4.11]{leiloefflerzerbes11}, this allows us to deduce the following description of the image of $\uCol$.
\begin{corollary}
 Let $\eta\in\hat\Delta$, then
 \[
 e_\eta\image(\uCol)\otimes E=\{\underline{F}\in\cO[[X]]\otimes E:\underline{F}(u^j-1)\in V_{\eta,j}, 0\le j\le k_f+k_g+1\}.
 \]
\end{corollary}

If  $\fS=\{(\ts),(\bc)\}\in \cS$, this corresponds to a two-dimensional subspace in $\Dcris(T)\otimes E$, generated by two elements of the basis $\{v_1,v_2,v_3,v_4\}$, which we denote by $V_\fS$. If we write $\col_\fS$ for the wedge product 
\[
\col_{\ts}\wedge\col_{\bc}:\bigwedge^2 \HIw(\Qp,T)\to \Lambda,
\]
then we may describe its image by
\[
e_\eta\image(\col_\fS)\otimes E=\prod_{j=0}^{k_f+k_g+1}(X-u^j+1)^{n_{\fS,\eta,j}}\cO[[X]]\otimes E,
\]
where $n_{\fS,\eta,j}=\dim_E V_\fS\cap V_{\eta,j}$. If we do not tensor our image by $E$, we have that $e_\eta\image(\col_\fS)$ is pseudo-isomorphic to $p^{\mu_{\fS,\eta}}\prod_{j=0}^{k_f+k_g+1}(X-u^j+1)^{n_{\fS,\eta,j}}\cO[[X]]$ for some integer $\mu_{\fS,\eta}$.

%%%%%%%%%%%%%%%%%%%%%%%%%%%%%%%%%%%%%%%%%%%%%%%%%%%%%%%%%%%%%%%%%%%%%%%%%%%%%%%%%%%%%%%%%%%%%%%%%%%%%%%%%%%%%%%%%%%%%%%%%%%%%%%%%%%%%%%%%%%%%%%%%%%%%%%%%%%%%%%%%%%%%%%%%%%%%%%%%%%%%%%%%%%%%%%%%%%%%%%%%%%%%%%%%%%%%%%%%%%%%%%%%%%%%%%%%%%%%%%%%%%%%%%%%%%%%%%%%%%%%%%%%%%%%%%%%%%%%%%%%%%%%%%%%%%%%%%%%%%%%%%%%%%%%%%%%%%%%%%%%%%%%%%%%%%%%%%%%%%%%%%%%%%%%%%%%%%%%%%%%%%%%%%%%%%%%%%%%%%%%%%%%%%%%%%%%%%%%%%%%%%%%%%%%%%%%%%%%%%%%%%%%%%%%%%%%%%%%%%%%%%%%%%%%%%%%%%%%%%%%%%%%%%%%%%%%%%%%%%%%%%%%%%%%%%%%%%%%%%%%%%%%%%%%%%%%%%%%%%%%%%%%%%%%%%%%%%%%%%%%%%%%%%%%%%%%%%%%%%%%%%%%%%%%%%%%%%%%%%%%%%%%%%%%%%%%%%%%%%%%%%%%%%%%%%%%%%%%%%%%%%%%%%%%%%%%%%%%%%%%%%%%%%%%%%%%%%%%%%%%%%%%%%%%%%%%%%%%%%%
\bibliographystyle{amsalpha}
\bibliography{references}

\providecommand{\bysame}{\leavevmode\hbox to3em{\hrulefill}\thinspace}
\providecommand{\MR}{\relax\ifhmode\unskip\space\fi MR }
% \MRhref is called by the amsart/book/proc definition of \MR.
\providecommand{\MRhref}[2]{%
  \href{http://www.ams.org/mathscinet-getitem?mr=#1}{#2}
}
\providecommand{\href}[2]{#2}
\begin{thebibliography}{BLV18}

\bibitem[AI17]{AIU}
Fabrizio Andreatta and Adrian Iovita, \emph{Triple product {$p$}-adic
  {$L$}-functions associated to finite slope {$p$}-adic families of modular
  forms}, preprint, with an appendix by Eric Urban. arxiv:1708.02785, 2017.

\bibitem[Ben15]{benoisSelmerComplexes}
Denis Benois, \emph{Selmer complexes and {$p$}-adic {H}odge theory},
  {Arithmetic Geometry}, LMS Lecture Note Series (420), Cambridge University
  Press, 2015, pp.~36--88.

\bibitem[Ber02]{berger02}
Laurent Berger, \emph{Repr\'esentations p-adiques et \'equations
  diff\'erentielles}, Invent. Math. \textbf{148} (2002), 219--284.

\bibitem[Ber03]{berger03}
\bysame, \emph{Bloch and {K}ato's exponential map: three explicit formulas},
  Doc. Math. (2003), no.~Extra Vol., 99--129, Kazuya Kato's fiftieth birthday.

\bibitem[Ber04]{berger04}
\bysame, \emph{Limites de repr\'esentations cristallines}, Compos. Math.
  \textbf{140} (2004), no.~6, 1473--1498.

\bibitem[Ber08]{berger08}
\bysame, \emph{Equations differentielles p-adiques et $(\varphi, {N})$-modules
  filtr\'es}, Ast\'erisque \textbf{319} (2008), 13--38.

\bibitem[BL15]{kbbleiPLMS}
K{\^a}z{\i}m B\"uy\"ukboduk and Antonio Lei, \emph{Coleman-adapted
  {R}ubin--{S}tark {K}olyvagin systems and supersingular {I}wasawa theory of
  {CM} abelian varieties}, Proc. Lond. Math. Soc. (3) \textbf{111} (2015),
  no.~6, 1338--1378.

\bibitem[BL17]{BL16b}
\bysame, \emph{{I}wasawa theory of elliptic modular forms over imaginary
  quadratic fields at non-ordinary primes}, 2017, submitted, arXiv:1605.05310.

\bibitem[BLV18]{BLV}
K{{\^a}}z{\i}m B{\"u}y{\"u}kboduk, Antonio Lei, and Guhan Venkat, \emph{Iwasawa
  theory for symmetric squares of {$p$}-non-ordinary eigenforms}, 2018,
  arxiv:1807.11517.

\bibitem[CC99]{cherbonniercolmez}
Fr\'ed\'eric Cherbonnier and Pierre Colmez, \emph{Th\'eorie d'{I}wasawa des
  repr\'esentations {$p$}-adiques d'un corps local}, J. Amer Math Soc.
  \textbf{12} (1999), 241--268.

\bibitem[Fon90]{fontaine90}
Jean-Marc Fontaine, \emph{Repr\'esentations p-adiques des corps locaux}, The
  Grothendieck Festschrift Volume 2 (Cartier et~al, ed.), vol.~87, Progress in
  Math. Birkhauser, 1990, pp.~249--309.

\bibitem[KLZ15]{KLZ1a}
Guido Kings, David Loeffler, and Sarah~Livia Zerbes, \emph{Rankin--{E}isenstein
  classes for modular forms}, arXiv:1501.03289, to appear in Amer. J. Math,
  2015.

\bibitem[KLZ17]{KLZ2}
Guido Kings, David Loeffler, and Sarah Zerbes, \emph{Rankin--{E}isenstein
  classes and explicit reciprocity laws}, Cambridge J. Math. \textbf{5} (2017),
  no.~1, 1--122.

\bibitem[Kob03]{kobayashi03}
Shin-ichi Kobayashi, \emph{Iwasawa theory for elliptic curves at supersingular
  primes}, Invent. Math. \textbf{152} (2003), no.~1, 1--36.

\bibitem[Lei11]{lei09}
Antonio Lei, \emph{Iwasawa theory for modular forms at supersingular primes},
  Compositio Math. \textbf{147} (2011), no.~03, 803--838.

\bibitem[Lei17]{leitohoku}
Antonio Lei, \emph{Bounds on the {T}amagawa numbers of a crystalline
  representation over towers of cyclotomic extensions}, Tohoku Math. J. (2)
  \textbf{69} (2017), no.~4, 497--524.

\bibitem[LLZ10]{leiloefflerzerbes10}
Antonio Lei, David Loeffler, and Sarah~Livia Zerbes, \emph{Wach modules and
  {I}wasawa theory for modular forms}, Asian J. Math. \textbf{14} (2010),
  no.~4, 475--528.

\bibitem[LLZ11]{leiloefflerzerbes11}
\bysame, \emph{Coleman maps and the {$p$}-adic regulator}, Algebra Number
  Theory \textbf{5} (2011), no.~8, 1095--1131.

\bibitem[LLZ14]{LLZ1}
\bysame, \emph{Euler systems for {R}ankin-{S}elberg convolutions of modular
  forms}, Ann. of Math. (2) \textbf{180} (2014), no.~2, 653--771.

\bibitem[LLZ17]{LLZCJM}
\bysame, \emph{On the asymptotic growth of
  {B}loch-{K}ato--{S}hafarevich--{T}ate groups of modular forms over cyclotomic
  extensions}, Canad. J. Math \textbf{69} (2017), 826--850.

\bibitem[Loe14]{Loeffler-Heidelberg}
David Loeffler, \emph{P-adic integration on ray class groups and non-ordinary
  p-adic {L}-functions}, Iwasawa Theory 2012: State of the Art and Recent
  Advances (Berlin) (Thanasis Bouganis and Otmar Venjakob, eds.), Contributions
  in Mathematical and Computational Sciences, vol.~7, Springer, 2014,
  pp.~401--441.

\bibitem[Loe17]{LoefflerGlasgow}
\bysame, \emph{Images of adelic {G}alois representations for modular forms},
  Glasg. Math. J. \textbf{59} (2017), no.~1, 11--25.

\bibitem[LZ15]{LZ2}
David Loeffler and Sarah~Livia Zerbes, \emph{Iwasawa theory for the symmetric
  square of a modular form}, 2015, \emph{ J. Reine Angew. Math.}, to appear.
  arXiv:1512.03678.

\bibitem[LZ16]{LZ1}
\bysame, \emph{Rankin--{E}isenstein classes in {C}oleman families}, Res. Math.
  Sci. \textbf{3} (2016), Paper No. 29, 53.

\bibitem[MR04]{mr02}
Barry Mazur and Karl Rubin, \emph{Kolyvagin systems}, Mem. Amer. Math. Soc.
  \textbf{168} (2004), no.~799, viii+96.

\bibitem[Nek06]{nekovarselmer}
Jan Nekov{\'a}{\v{r}}, \emph{Selmer complexes}, Ast\'erique, vol. 310, Soc.
  Math. France, 2006.

\bibitem[Ots09]{OTS09}
Rei Otsuki, \emph{Construction of a homomorphism concerning {E}uler systems for
  an elliptic curve}, Tokyo J. Math. \textbf{32} (2009), no.~1, 253--278.

\bibitem[Pot12]{jaycyclo}
Jonathan Pottharst, \emph{{C}yclotomic {I}wasawa theory of motives}, 2012,
  unpublished, available from \url{https://vbrt.org/writings/cyc.pdf}.

\bibitem[Pot13]{jayanalyticselmer}
\bysame, \emph{Analytic families of finite-slope {S}elmer groups}, Algebra
  Number Theory \textbf{7} (2013), no.~7, 1571--1612.

\bibitem[PR94]{perrinriou94}
Bernadette Perrin-Riou, \emph{Th\'eorie d'{I}wasawa des repr\'esentations
  {$p$}-adiques sur un corps local}, Invent. Math. \textbf{115} (1994), no.~1,
  81--161.

\bibitem[PR98]{PR98}
\bysame, \emph{Syst\`emes d'{E}uler {$p$}-adiques et th\'eorie d'{I}wasawa},
  Ann. Inst. Fourier (Grenoble) \textbf{48} (1998), no.~5, 1231--1307.

\bibitem[Rub96]{Rub96}
Karl Rubin, \emph{A {S}tark conjecture ``over {$\mathbf{Z}$}'' for abelian
  {$L$}-functions with multiple zeros}, Ann. Inst. Fourier (Grenoble)
  \textbf{46} (1996), no.~1, 33--62.

\end{thebibliography}
\end{document}